\documentclass[12pt]{amsart}
\usepackage{amssymb}
\usepackage{amsmath, amsthm, graphics}

\allowdisplaybreaks[3]

\usepackage{latexsym}
\usepackage{epsfig}
\usepackage{amsfonts}
\usepackage{enumerate}
\usepackage{times}
\usepackage{tikz-cd}

\newcommand{\excise}[1]{}

\newtheorem{theorem}{Theorem}

\newtheorem{lemma}[theorem]{Lemma}
\newtheorem*{thm*}{Theorem}
\newtheorem*{lem*}{Lemma}
\newtheorem*{prp*}{Proposition}

\newtheorem{proposition}[theorem]{Proposition}
\theoremstyle{definition}

\newtheorem{remark}[theorem]{Remark}

\newtheorem*{remark*}{Remark}

\def\<{\left\langle}
\def\>{\right\rangle}

\newcommand{\ee}{\mathrm{e}} 
\newcommand{\eem}{\varepsilon} 

\newcommand{\exterior}{\textstyle{\bigwedge}}

\newcommand{\Mbar}{\overline{M}}
\newcommand{\ev}{\mathrm{ev}}

\newcommand{\isom}{\cong}
\renewcommand{\tilde}{\widetilde}

\newcommand{\qshift}[1]{q^{Q_{#1}\partial_{Q_{#1}}}}
\newcommand{\Acom}[1]{\mathsf{A}_{#1,\mathrm{com}}}
\newcommand{\Aa}{\mathsf{A}}
\newcommand{\TT}{\mathsf{T}}
\newcommand{\MM}{\mathsf{M}}
\newcommand{\UU}{\mathsf{U}}

\renewcommand{\AA}{\mathbb{A}}
\newcommand{\PP}{\mathbb{P}}
\newcommand{\CC}{\mathbb{C}}

\newcommand{\Hom}{\mathrm{Hom}}
\newcommand{\Sym}{\mathrm{Sym}}

\newcommand{\OO}{\mathcal{O}}
\newcommand{\QM}{\mathcal{Q}} 
\newcommand{\Zas}{\mathcal{Z}} 

\newcommand{\sfT}{\mathsf{T}} 
\newcommand{\sfA}{\mathsf{A}} 
\newcommand{\sfS}{\mathsf{S}} 
\newcommand{\R}{\mathbb{R}} 
\newcommand{\Z}{\mathbb{Z}} 
\newcommand{\C}{\mathbb{C}} 

\newcommand{\Eff}{\operatorname{Eff}} 
\newcommand{\ord}{\operatorname{ord}} 
\newcommand{\End}{\operatorname{End}}

\newcommand{\corr}[1]{\left\langle #1 \right\rangle} 
\newcommand{\parfrac}[2]{\frac{\partial #1}{\partial #2}} 

\begin{document}

\title{On the finiteness of quantum K-theory of a homogeneous space}

\date{February 12, 2020}
\author[Anderson]{David Anderson}
\address{Department of Mathematics\\ Ohio State University\\ 100 Math Tower, 231 West 18th Ave. \\ Columbus,  OH 43210\\ USA}
\email{anderson.2804@math.osu.edu}

\author[Chen]{Linda Chen}
\address{Department of Mathematics and Statistics\\ Swarthmore College\\ Swarthmore, PA 19081\\ USA}
\email{lchen@swarthmore.edu}

\author[Tseng]{Hsian-Hua Tseng}
\address{Department of Mathematics\\ Ohio State University\\ 100 Math Tower, 231 West 18th Ave. \\ Columbus,  OH 43210\\ USA}
\email{hhtseng@math.ohio-state.edu}

\thanks{D. A. is supported in part by NSF grant DMS-1502201. L. C. is supported in part by Simons Foundation Collaboration grant 524354. H.-H. T. is supported in part by NSF grant DMS-1506551.}

\begin{abstract}
We show that the product in the quantum K-ring of a generalized flag manifold $G/P$ involves only finitely many powers of the Novikov variables.  In contrast to previous approaches to this finiteness question, we exploit the finite difference module structure of quantum K-theory.  At the core of the proof is a bound on the asymptotic growth of the $J$-function, which in turn comes from an analysis of the singularities of the zastava spaces studied in geometric representation theory.

An appendix by H.~Iritani establishes the equivalence between finiteness and a quadratic growth condition on certain shift operators.
\end{abstract}

\maketitle

Let $G$ be a simply connected complex semisimple group, with Borel subgroup $B$, maximal torus $T$, and standard parabolic group $P$. The main aim of this article is to prove a fundamental fact about the quantum K-ring of the homogeneous space $G/P$.

\begin{thm*}
The structure constants for (small) quantum multiplication of Schubert classes in $QK_T(G/P)$ are polynomials in the Novikov variables, with coefficients in the representation ring of the torus.
\end{thm*}

This is proved as Theorem~\ref{t.finite} below.  A priori, quantum structure constants are power series in the Novikov variables, which keep track of degrees of curves; our theorem says that in fact, only finitely many degrees appear.  This property is often referred to as {\em finiteness} of the quantum product.

Finiteness has been the subject of conjectures since the beginnings of the combinatorial study of quantum K-theory in Schubert calculus.  Indeed, this property is a foundational prerequisite for the main components of Schubert calculus: a presentation of the quantum K-ring as a quotient by a polynomial ring; a set of polynomial representatives for Schubert classes; and finally, combinatorial formulas for the structure constants themselves.  

In quantum cohomology, finiteness of the quantum product is immediate from the definition.  In this case, the structure constants are Gromov-Witten invariants---certain integrals on the moduli space of stable maps into $G/P$---and they automatically vanish for curves of sufficiently large degree, by dimension reasons.  In K-theory, by contrast, the analogous Gromov-Witten invariants are certain Euler characteristics on the moduli space, and there is no reason for them to vanish for large degrees---in fact they do not.   The structure constants for the quantum product in K-theory are rather complicated alternating sums of Gromov-Witten invariants, so a direct proof of finiteness involves demonstrating massive cancellation.

In the cases where finiteness was previously known, this direct approach was used, employing a detailed analysis of the geometry of the moduli space of stable maps, and especially its ``Gromov-Witten subvarieties'', whose Euler characteristics compute K-theoretic Gromov-Witten invariants of $G/P$.  In their paper on Grassmannians, Buch and Mihalcea showed that these Gromov-Witten varieties are rational for sufficiently large degrees; this implies that the corresponding invariants are equal to $1$, and the required cancellation can be deduced combinatorially \cite{bm}.  Together with Chaput and Perrin, they extended this idea to prove finiteness for {\em cominuscule varieties}, a certain class of homogeneous varieties of Picard rank one \cite{bcmp1,bcmp2}.  (Furthermore, according to \cite[Remark 1.1]{bcmp2}, finiteness holds for any projective rational homogeneous space of Picard rank one.)

Recently, Kato \cite{k1, k2} has proven some remarkable conjectures \cite{llms} about the quantum K-ring of a {\em complete} flag variety $G/B$.  Up to inverting some elements, he establishes ring isomorphisms
\[
  QK_T(G/B) \isom K_T^\circ( \text{ semi-infinite flag variety } ) \isom K^T_\circ( \text{ affine Grassmannian }).
\]
In particular, Kato's work implies finiteness for $QK_T(G/B)$.  See \cite[Corollary~3.3]{k1}, noting that the argument given there relies on our Lemma~\ref{l.Acom} (in establishing the first isomorphism above), but otherwise is independent of our approach.

In this paper we prove the finiteness result for $QK_T(G/P)$ for all partial flag varieties. The starting point of our method is the fundamental fact that quantum K-theory admits the structure of a {\em $\mathsf{D}_q$-module}. This structure was first found for the quantum K-theory of the complete flag variety $Fl_{r+1}=SL_{r+1}/B$ by Givental and Lee, and later derived in general by Givental and Tonita from a characterization theorem of quantum K-theory in terms of quantum cohomology, the so-called {\em quantum Hirzebruch-Riemann-Roch theorem} \cite{gl,gt}. As explained by Iritani, Milanov, and Tonita, this $\mathsf{D}_q$-module structure is manifested as a difference equation (Equation (\ref{dq2}) below) satisfied by certain generating functions $J$ and $\TT$ of K-theoretic Gromov-Witten invariants; they also explain how the quantum product by a line bundle is related to these generating functions and use this to compute the quantum product for $Fl_3$ \cite{imt}.  More details are reviewed in \S\ref{subsec:J}.  

The general strategy of our proof can be summarized as follows.  If one can appropriately bound the coefficients appearing in the generating functions $J$ and $\TT$, then results of \cite{imt} allow one to deduce that the quantum product by a line bundle is finite.  For a complete flag variety, this is sufficient, since $K_T(G/B)$ is generated by line bundles.  In fact, it is also true that the K-theory of $G/P$ is generated by line bundle classes, after inverting certain elements of the representation ring; this seems to be less well known, so we include a proof in Lemma~\ref{l.generate}.

The technical heart of our argument lies in obtaining the appropriate bound on the growth of coefficients of $J$ and $\mathsf{T}$ as $q\to +\infty$.  Here we divide the problem and treat the $G/B$ and $G/P$ cases separately.  For $G/B$, we analyze the geometry of the {\em zastava space}, a compactification of the space of (based) maps studied extensively in geometric representaion theory.  Specifically, we use a computation of the canonical sheaf of the zastava space due to Braverman and Finkelberg \cite{bf1,bf2}, together with some properties of its singularities.  This leads to the bound for $J$ stated in Lemma~\ref{l.bound}, as well as the stronger bound of Lemma~\ref{l.boundADE} for simply-laced types.  For bounds for $\mathsf{T}$ we appeal to Kato's work and a result of H. Iritani (the Proposition of Appendix \ref{app:iritani}). We then transfer our bounds for $G/B$ to bounds for $G/P$, using the main geometric constructions in Woodward's proof of the Peterson comparison formula \cite{w}.

With the bounds in hand, we deduce finiteness in \S\ref{sec:finite}.  Here our arguments take advantage of the explicit form of our bounds for $J$, together with an inequality in root lattices proved in Appendix \ref{app:ineq}.

We expect our methods to find further applications in quantum Schubert calculus.  Most immediately, we can establish a presentation of the quantum K-ring of $SL_{r+1}/B$, resolving a conjecture by Kirillov and Maeno \cite{lm,iim}.  (Using a different definition of quantum K-theory, a similar presentation was obtained in \cite{kpsz}.)  Together with algebraic work done by Ikeda, Iwao, and Maeno \cite{iim}, this confirms some conjectural relations between the quantum K-ring of the flag manifold and the K-homology of the affine Grassmannian \cite{llms}, giving an alternative to Kato's approach.  Some results in this direction are included in our preprint \cite{act}.

A secondary goal of this article is to illustrate the power of finite-difference methods in quantum Schubert calculus.  To this end, we have included a fair amount of background.  We hope these sections may serve as a helpful companion to the foundational  papers of Givental and others.

\bigskip
\noindent
{\it Acknowledgements}. 
We thank A.~Givental, T.~Ikeda, H.~Iritani, S.~Kato, S.~Kov\'acs, and C.~Li for helpful discussions, and the referees for insightful comments that improved the manuscript.

\section{Background}

\subsection{Roots and weights}

Let $\Lambda$ be the weight lattice for the torus $T$, and let $\varpi_1,\ldots,\varpi_r$ be the fundamental weights for the Lie algebra of $G$.  The representation ring $R(T)$ is naturally identified with the group ring $\mathbb{Z}[\Lambda]$, and can be written as a Laurent polynomial ring $\mathbb{Z}[\ee^{\pm \varpi_1},\ldots,\ee^{\pm\varpi_r}]$.  The simple roots $\alpha_1,\ldots,\alpha_r$ generate a sublattice of $\Lambda$.  The coroot lattice $\check\Lambda$ has a basis of simple coroots $\check\alpha_1,\ldots,\check\alpha_r$, dual to $\varpi_1,\ldots,\varpi_r$.  We often write
\[
  \lambda = \lambda_1 \varpi_1 + \cdots + \lambda_r \varpi_r \quad \text{ and } \quad d = d_1\check\alpha_1+\cdots+d_r\check\alpha_r
\]
for elements of $\Lambda$ and $\check\Lambda$.  Usually, $d$ denotes a {\em positive element} of the coroot lattice, meaning all the integers $d_i$ are nonnegative.  We write $d\geq 0$ or $d\in \check\Lambda_+$ to indicate positive elements, and $d>0$ to mean a nonzero such $d$.  This induces a partial order in the standard way, so $d'\geq d$ iff $d'-d\geq 0$; that is, $d'_i\geq d_i$ for all $i$.
 
The vector spaces $\Lambda\otimes\mathbb{R}$ and $\check\Lambda\otimes\mathbb{R}$ are identified using the Weyl-invariant inner product $(\;,\;)$, normalized so that $(\alpha_i,\alpha_i)=2$ when $\alpha_i$ is a short root.  For example, this means $(d,\lambda) = \sum d_i\lambda_i$.  For $G=SL_{r+1}$, we have
\[
 (d,d) = \sum_{i=1}^{r+1} (d_i-d_{i-1})^2,
\]
where by convention $d_0=d_{r+1}=0$.

A {\em standard parabolic subgroup} is a closed subgroup $P$ such that $G\supseteq P\supseteq B$.  By recording which negative simple roots occurs as weights on the Lie algebra of $P$, such parabolics correspond to subsets of the simple roots.  (To be clear, $B$ corresponds to the empty set, while $G$ corresponds to the whole set of simple roots.)  Let $I_P\subseteq \{1,\ldots,r\}$ be the indices of simple roots corresponding to $P$.

The sublattice $\Lambda_P\subseteq \Lambda$ of weights $\lambda$ such that $(\check\alpha_i,\lambda)=0$ for $i\in I_P$ is spanned by the weights $\varpi_j$ for $j\not\in I_P$.  Dually, $\check\Lambda_P \subseteq \check\Lambda$ is the sublattice spanned by $\check\alpha_i$ for $i\in I_P$.  We write $\check\Lambda^P = \check\Lambda/\check\Lambda_P$, and $\check\Lambda^P_+$ for the image of $\check\Lambda_+$.  So $\check\Lambda^P_+$ is spanned by the images of $\check\alpha_i$ for $i\not\in I_P$.

Let $\rho=\varpi_1+\cdots+\varpi_r$ be the {\em Weyl element}, the smallest regular dominant weight.  For any $d\in\check\Lambda$, we have  $(d,\rho)=\sum d_i =: |d|$.

\subsection{Flag varieties}\label{subsec:fl}

Each weight $\lambda\in \Lambda$ gives rise to an equivariant line bundle $P^\lambda$ on the complete flag variety $G/B$.  Writing $P_i$ for the line bundle corresponding to $\varpi_i$, we have $P^\lambda = P_1^{\lambda_1}\cdots P_r^{\lambda_r}$ when $\lambda = \lambda_1 \varpi_1 + \cdots + \lambda_r \varpi_r$.

Each fundamental weight $\varpi_i$ corresponds to an irreducible representation $V_{\varpi_i}$.  There is an embedding
\[
\iota\colon G/B \hookrightarrow \Pi := \prod_{i=1}^r \mathbb{P}(V_{\varpi_i}),
\]
such that $P_i = \iota^*\mathcal{O}_i(-1)$ is the pullback of the tautological subbundle from the $i$th factor of $\Pi$.

For example, when $G=SL_{r+1}$, the flag variety $G/B=Fl_{r+1}$ parametrizes all complete flags in $\mathbb{C}^{r+1}$.  We have $V_{\varpi_i}=\exterior^i\mathbb{C}^{r+1}$, and the line bundle $P_i$ is the top exterior power $\exterior^iS_i$ of the $i$th tautological subbundle on $Fl_{r+1}$.\footnote{Our conventions agree with \cite{gl}, but are opposite to those of \cite{imt}, where $P_i$ is replaced by $P_i^{-1}$.}

Equivariant line bundles on $G/P$ correspond to weights $\lambda\in \Lambda_P$.  We will continue to use the notation $P^\lambda$ for such bundles; the meaning of ``$P$'' (as parabolic or line bundle) should be clear from context. 
As with $G/B$, there is an embedding
\[
\iota\colon G/P \hookrightarrow  \prod_{j\not\in I_P} \mathbb{P}(V_{\varpi_j}),
\]
with $P_j$ being the pullback of $\mathcal{O}(-1)$ from the $j$th factor.

There are natural identifications $H_2(G/B,\mathbb{Z}) = \check\Lambda$ and $\mathrm{Eff}_2(G/B) = \check\Lambda_+$, as well as $H_2(G/P,\mathbb{Z}) = \check\Lambda^P$ and $\mathrm{Eff}_2(G/P) = \check\Lambda^P_+$.  
The pushforward $H_2(G/B) \to H_2(G/P)$ is identified with the quotient map $\check\Lambda \to \check\Lambda^P$.  The pullback $H^2(G/P)\to H^2(G/B)$ dual to this projection is identified with the inclusion $\Lambda_P \hookrightarrow \Lambda$.

It is a basic fact that $K_T(G/B)$ is generated by $P_1,\ldots,P_{r}$ as an $R(T)$-algebra; that is, there is a surjective homomorphism
\[
  R(T)[P_1,\ldots,P_{r}] \twoheadrightarrow K_T(G/B).
\]
(See, for example, \cite[\S4]{kk}.)  Thus there is an $R(T)$-basis for $K_T(G/B)$ consisting of monomials in the $P_i$, and in particular, any other basis---for example, a Schubert basis---can be written as a finite $R(T)$-linear combination of such monomials.

In general, it is not the case that $K_T(G/P)$ is generated by line bundles as an $R(T)$-algebra.  However, after extending scalars to the fraction field $F(T)$ of $R(T)$, the algebra is generated by line bundles.  This fact seems to be less well known, although it is implicit in \cite{bcmp3}, and the idea of the proof can be found in \cite[Lemma~4.1.3]{cfks}.  For clarity, we state a general version here.

\begin{lemma}\label{l.generate}
Let $X \hookrightarrow Y$ be a closed $T$-equivariant inclusion of smooth varieties.  Assume that the restriction homomorphism $K_T(Y^T) \to K_T(X^T)$ is surjective.  If $\{\alpha\}$ is a set of generators for $K_T(Y)$ as an $R(T)$-algebra, then the restrictions $\{\beta\}$ generate $F(T)\otimes_{R(T)} K_T(X)$ as an $F(T)$-algebra.
\end{lemma}

\begin{proof}
The proof follows directly from the localization theorem, which gives natural isomorphisms $F(T)\otimes_{R(T)} K_T(X)  \isom F(T)\otimes_{R(T)} K_T(X^T)$.  A little more precisely, rather than passing to $F(T)$, it suffices to invert elements $1-\ee^{-\alpha}$ of $R(T)$, where $\alpha$ runs over characters appearing in the normal spaces to $X^T$ in $X$.
\end{proof}

A particular case of the lemma is this: 
\begin{quote}
{\em Whenever $X$ is a smooth projective variety with finitely many attractive fixed points, the $F(T)$-algebra $F(T)\otimes_{R(T)} K_T(X)$ is generated by the class of an ample line bundle.}
\end{quote}
An isolated fixed point $p$ of a (possibly singular) variety $X$ is called {\em attractive} if all the weights of the action of $T$ on the Zariski tangent space at $p$ lie in an open half-space.  This condition guarantees that under any equivariant embedding $X \hookrightarrow \PP^n$, each of the finitely many fixed points of $X$ maps to a distinct connected component of $(\PP^n)^T$, which in turn implies that the restriction map is surjective.

The standard torus action on $G/P$ has finitely many attractive fixed points, so the lemma applies to the case we study.  (A different, combinatorial argument for equivariant cohomology of $G/P$ is given in \cite[Remark~5.11]{bcmp3}.)

\subsection{Equivariant multiplicities and the fixed-point formula}\label{ss.localize}

One of the main tools for computing in quantum K-theory is torus-equivariant localization on moduli spaces.  We quickly review the main theorem we will use.  This material is standard; see, e.g., \cite{a} for an exposition aligned with our needs, \cite{brion} for a parallel discussion in the case of equivariant Chow groups, and \cite{behrend} for applications to Gromov-Witten theory.

Suppose a torus $T$ acts on a variety $X$.  The Grothendieck group of equivariant coherent sheaves is $K^T_\circ(X)$.  
There is a natural isomorphism
\begin{equation}\label{e.local-isom}
  F(T)\otimes_{R(T)} K^T_\circ(X^T) \xrightarrow{\sim} F(T)\otimes_{R(T)} K^T_\circ(X)
\end{equation}
induced by pushforward from the fixed locus.  (This goes back to Atiyah \cite{atiyah} and Quart \cite{quart}.)  Since $T$ acts trivially on $X^T$, the left-hand side is
\[
  F(T)\otimes_{R(T)} K^T_\circ(X^T) = F(T)\otimes_\mathbb{Z} K_\circ(X^T) = \bigoplus_{Z\subseteq X^T} F(T)\otimes_\mathbb{Z} K_\circ(Z),
\]
the direct sum over connected components $Z\subseteq X^T$.

If $Z\subseteq X^T$ is a connected component, the {\em equivariant multiplicity} of $X$ along $Z$ is the element $\eem_Z(X)$ of $F(T)\otimes_\mathbb{Z} K_\circ(Z)$ defined so that
\[
  \sum_{Z\subseteq X^T} \eem_Z(X) = [\mathcal{O}_X]
\]
under the isomorphism \eqref{e.local-isom}.  More generally, the localization isomorphism respects products by vector bundles: given a class $\xi\in K_T^\circ(X)$ (the Grothendieck group of equivariant vector bundles), one has
\begin{equation}\label{e.restrict}
  \sum_{Z\subseteq X^T} \eem_Z(X)\cdot \xi|_Z = \xi \cdot [\mathcal{O}_X].
\end{equation}
Here $(\cdot)|_Z$ denotes the restriction homomorphism $K_T^\circ(X) \to K_T^\circ(Z)$.

The localization isomorphism is natural in an evident way: if $\pi\colon X \to Y$ is a proper equivariant morphism, then there is a commuting square
\[
\begin{tikzcd}
  F(T)\otimes_{R(T)} K^T_\circ(X^T) \ar[r,"\sim"] \ar[d,"\pi_*"] & F(T)\otimes_{R(T)} K^T_\circ(X) \ar[d,"\pi_*"] \\
  F(T)\otimes_{R(T)} K^T_\circ(Y^T) \ar[r,"\sim"] & F(T)\otimes_{R(T)} K^T_\circ(Y).
\end{tikzcd}
\]
Naturality immediately implies a useful formula for equivariant multiplicities.  Assume $\pi_*[\mathcal{O}_X] = [\mathcal{O}_Y]$.  (For example, this holds if $X$ and $Y$ both have rational singularities and $\pi$ is birational, or has connected rational fibers.)  Then for any connected component $W\subseteq Y^T$, we have the formula
\begin{equation}\label{e.fpformula}
  \eem_W(Y) = \sum_{Z\subseteq (\pi^{-1}W)^T} \pi^Z_*  \eem_Z(X) ,
\end{equation}
the sum over connected components $Z\subseteq X^T$ which map into a given connected component $W\subseteq Y^T$, where $\pi^Z\colon Z \to W$ is the restriction of $\pi$.  This gives a means of computing the equivariant multiplicities.

If the connected component $Z\subseteq X^T$ is regularly embedded, with conormal bundle $N_{Z/X}^*$, then the equivariant multiplicity is
\begin{equation}\label{e.em}
  \eem_Z(X) = \frac{1}{\lambda_{-1}(N^*_{Z/X})}.
\end{equation}
Here the denominator is the K-theoretic Euler class of $N_{Z/X}$.  (More generally, for any vector bundle $E$ of rank $e$, one defines
\[
  \lambda_{-1}(E) = 1 - E + \exterior^2 E - \cdots +(-1)^e \exterior^e E.)
\]

Suppose $\pi\colon X \to Y$ is a proper equivariant map, and $W\subseteq Y^T$ is a connected component which is regularly embedded, such that all components $Z\subseteq (\pi^{-1}W)^T$ are regularly embedded in $X$.  (For example, this happens automatically if $X$ and $Y$ are nonsingular varieties.)  Combining \eqref{e.restrict}, \eqref{e.fpformula}, and \eqref{e.em}, we have
\begin{equation}\label{e.fp2}
 \frac{(\pi_*\xi)|_{W}}{\lambda_{-1}(N^*_{W/Y})} = \sum_{Z\subseteq (\pi^{-1}W)^T} \pi^Z_*\left( \frac{\xi|_Z}{\lambda_{-1}(N^*_{Z/X})} \right),
\end{equation}
for any element $\xi\in K^T_\circ(X) = K_T^\circ(X)$,

A simple special case of the equivariant multiplicity will be of particular interest to us.  When $X$ is affine, and $Z=p$ is any fixed point, the equivariant multiplicity is equal to the {\em graded character} $\mathrm{ch}(\mathcal{O}_X)$ (see, e.g., \cite{rossmann}).  If, furthermore, the fixed point is {\em attractive}, the equivariant multiplicity is equal to the multigraded Hilbert series of $\mathcal{O}_X$.  (For example, if $T$ acts on $X=\mathbb{A}^1$ by the character $\ee^\alpha$, then it acts on $\mathcal{O}_X = \mathbb{C}[x]$ by scaling $x$ by $\ee^{-\alpha}$, so we have $\eem_0(X) = \mathrm{ch}(\mathcal{O}_X) = 1/(1-\ee^{-\alpha})$.)

\subsection{Quantum K-theory and moduli spaces}\label{subsec:moduli}

The (genus $0$) K-theoretic Gromov-Witten invariants are defined as certain sheaf Euler characteristics on the space of $n$-pointed, degree $d$ stable maps,
\[
  \Mbar_{0,n}(G/P,d).
\]
This space comes with evaluation morphisms $\ev_i\colon \Mbar_{0,n}(G/P,d) \to G/P$ for $1\leq i\leq n$, which are equivariant for the action of $T$ on $G/P$ and the induced action on $\Mbar_{0,n}(G/P,d)$.  Given classes $\Phi_1,\ldots,\Phi_n\in K_T(G/P)$, there is a Gromov-Witten invariant
\[
  \chi(\Mbar_{0,n}(G/P,d),\,\ev_1^*\Phi_1 \cdots \ev_n^*\Phi_n) \in R(T).
\]

The {\em Novikov variables} keep track of curve classes in $G/P$; for $d\in \check\Lambda^P_+$, we write $Q^d = Q_1^{d_1}\cdots Q_r^{d_r}$.  The (small) quantum K-ring of $G/P$ is defined additively as
\[
  QK_T(G/P) := K_T(G/P) \otimes \mathbb{Z}[[Q]],
\]
and is equipped with a {\em quantum product} $\star$ which deforms the usual (tensor) product on $K_T(G/P)$.  Choosing any $R(T)$-basis\footnote{The classes $\Phi_w$ are not necessarily Schubert classes; in fact, after extending scalars from $R(T)$ to $F(T)$, we will use a monomial basis consisting of certain $P^\lambda$'s.} $\{\Phi_w\}$ for $K_T(G/P)$, and using the same notation for the corresponding $R(T)[[Q]]$-basis for $QK_T(G/P)$, one has
\[
  \Phi_u \star \Phi_v = \sum_{w,d} N_{u,v}^{w,d} Q^d \Phi_w,
\]
where a priori the right-hand side is an infinite sum over all $d\in \check\Lambda^P_+$.  (The structure constants $N_{u,v}^{w,d}$ are defined in a rather involved way via alternating sums of Gromov-Witten invariants; see \cite{g,l,bm} for details.)

We work mainly with two compactifications of the space  $\Hom_d(\PP^1,G/P)$ of degree $d$ maps from $\mathbb{P}^1$ to $G/P$.  The first is Drinfeld's {\em quasimap space} $\QM_d$, and we use it only for $G/B$.  This space may be defined as follows; see, e.g., \cite{b-icm} for more details.  For projective space $\mathbb{P}(V)$ and an integer $d_i\geq 0$, let $\PP(V)_{d_i} = \PP(\mathrm{Sym}^{d_i}\mathbb{C}^2\otimes V)$ be the projective space of $V$-valued binary forms of degree $d_i$.  (This is the quot scheme compactification of the space of degree $d$ maps $\mathbb{P}^1 \to \mathbb{P}(V)$.)  With $\Pi = \prod_{i=1}^r \PP(V_{\varpi_i})$ as above and $d\in \check\Lambda_+$, let $\Pi_d = \prod_{i=1}^r\PP(V_{\varpi_i})_{d_i}$.  This contains the space of maps $\Hom_d(\PP^1,\Pi)$ as an open subset.  The embedding $\iota \colon G/B \hookrightarrow \Pi$ induces an embedding $\Hom_d(\PP^1,G/B) \hookrightarrow \Hom_d(\PP^1,\Pi)$, and the quasimap space $\QM_d$ is the closure of $\Hom_d(\PP^1,G/B)$ inside $\Pi_d$.

Spaces of maps and quasimaps are equipped with a $\mathbb{C}^*$-action induced from an action on the source curve.  The action on $\mathbb{P}^1$ is given by $q\cdot [a,b] = [a,qb]$, where $q$ is a coordinate on $\mathbb{C}^*$, so the fixed points are $0=[1,0]$ and $\infty=[0,1]$.  The $\mathbb{C}^*$-fixed loci in $\Pi_d$ are easy to describe: for each expression $d=d^-+d^+$ (with $d^-,d^+\in\check\Lambda_+$), there is a fixed component $\Pi_d^{(d^+)}$ consisting of tuples of monomials of bidegree $(d_i^-,d_i^+)$ on the factor $\mathbb{P}(V_{\varpi_i})_{d_i}$.  Using monomials to denote weight bases for $\mathrm{Sym}^{d_i}\mathbb{C}^2$, we have
\[
  \Pi_d^{(d^+)} = \prod_{i=1}^r \mathbb{P}( x_i^{d_i^-}y_i^{d_i^+} \otimes V_{\varpi_i}),
\]
so each such component is isomorphic to $\Pi$ itself.  

The $\mathbb{C}^*$-fixed components of $\QM_d\subseteq \Pi_d$ are $\QM_d^{(d^+)}\subseteq \Pi_d^{(d^+)}$, each isomorphic to $G/B \subseteq \Pi$.  

If we also consider the action of $T$ induced from the target space $G/B$, the quasimap space $\QM_d$ has finitely many $\mathbb{C}^*\times T$-fixed points, indexed by $(d^+,w)$ as $w$ ranges over the Weyl group.

Our second compactification of the space of maps is the {\em graph space},
\[
  \Gamma (G/P)_d := \Mbar_{0,0}(\PP^1\times G/P, (1,d) ).
\]
It includes $\Hom_d(\PP^1,G/P)$ as the open subset of stable maps with irreducible source, regarded as the graph of a map $\PP^1\to G/P$.  This space also comes with an action of $\mathbb{C}^*\times T$, induced from the componentwise action on $\PP^1\times G/P$.  As explained in \cite[\S2.2]{gl} and \cite[\S2.6]{imt}, the $\mathbb{C}^*$-fixed components of $\Gamma (G/P)_d$ correspond to certain maps where the source curve is reducible.  For each decomposition $d=d^-+d^+$, there is a component $\Gamma (G/P)_d^{(d^+)}$ whose general points parametrize maps with source curve having three components: a ``horizontal'' $\PP^1$ with degree $0$ with respect to $G/P$; a ``vertical'' $\PP^1$ attached to the first component at the fixed point $0$, with $G/P$-degree $d^+$; and a ``vertical'' $\PP^1$ attached to the first component at $\infty$, with $G/P$-degree $d^-$.  (If $d^+$ or $d^-$ is zero, the corresponding component of the source curve is absent.)  There are also pointed versions of graph spaces, $\Gamma (G/P)_{n,d}$, with $n\geq0$ marked points, defined as $\Mbar_{0,n}(\PP^1\times G/P, (1,d))$.  The fixed loci of these pointed spaces are similar, with the marked points being allocated to one of the two vertical curves.

There is a birational morphism $\mu\colon \Gamma (G/B)_d \to \QM_d\subseteq \Pi_d$, described in \cite[\S3]{gl}, and 

the fixed component $\Gamma (G/B)_d^{(d^+)}$ maps onto $\QM_d^{(d^+)}$ under $\mu$.  There are also morphisms $\beta_n\colon \Gamma (G/P)_{n,d} \to \Mbar_{0,n}(G/P,d)$, which, composed with evaluation morphisms from $\Mbar_{0,n}(G/P,d)$ to $G/P$, give morphisms $\ev_i\colon \Gamma (G/P)_{n,d}\to G/P$, for $1\leq i\leq n$.  

A key property of each of these moduli spaces---$\Mbar_{0,n}(G/P,d)$, $\Gamma (G/P)_{n,d}$, and $\QM_d$---is that they have rational singularities.  (For the first two, this is a general fact about varieties with finite quotient singularities; for $\QM_d$, it is one of the main theorems of \cite{bf1,bf2}.)  We  exploit this to freely transport computations of Euler characteristics from one of these spaces to another.

\subsection{The $J$-function and $\mathsf{D}_q$-module structure}\label{subsec:J}

The structure of quantum K-theory becomes clearer when Gromov-Witten invariants are packaged into a generating function, the {\em $J$-function}. Note that the definitions of $J$ vary somewhat in the literature.  Ours is that of \cite{gl}; the function of \cite{imt} is equal to our $(1-q)J$.  The function of \cite{bf1} is a certain localization of our $J$-function.  This function satisfies a finite-difference equation, and it is this $\mathsf{D}_q$-module structure we exploit to prove finiteness of the quantum product.  Here we review the properties of the $J$-function which we will need.  In this subsection, $X$ may be any smooth projective variety, as considered in \cite{imt}.

Consider the evaluation morphism $\ev\colon \Mbar_{0,1}(X,d) \to X$, which is equivariant for $\mathbb{C}^*\times T$ (with $\mathbb{C}^*$ acting trivially on both $\Mbar_{0,1}(X,d)$ and $X$).  The $J$-function is a power series in $Q$, with coefficients in $K_{T}(X)\otimes\mathbb{Q}(q)$:
\begin{equation}\label{e.jdef}
 J := 1 + \frac{1}{1-q} \sum_{d>0} Q^d\,\ev_*\left( \frac{1}{1-qL} \right).
\end{equation}
Here the character $q$ identifies $K_{\mathbb{C}^*}(\mathrm{pt})=\mathbb{Z}[q^{\pm}]$, and $L$ is the cotangent line bundle on $\Mbar_{0,1}(X,d)$.   (Its fiber at a moduli point $[f:(C,p)\to X]$ is $T_p^*C$.)  We often write
\[
  J = \sum_{d\geq 0} J_d\,Q^d,
\]
with $J_d\in K_{T}(X)\otimes\mathbb{Q}(q)$.

In \cite{imt}, a {\em fundamental solution} $\TT$ is defined.  This is an element of $\mathrm{End}_{R(T)}(K_T(X))\otimes \mathbb{Q}(q)[[Q]]$, and is characterized by
\begin{equation}
  \chi(X,\,\Phi_u\cdot  \TT(\Phi_v) ) = \chi(X, \Phi_u\cdot\Phi_v) + \sum_{d>0} Q^d \chi\left(\Mbar_{0,2}(X,d), \, \ev_1^*\Phi_u\cdot \frac{1}{1-q L_1}  \cdot  \ev_2^*\Phi_v \right),
\end{equation}
for all $\Phi_u$ and $\Phi_v$ in an $R(T)$-basis of $K_T(X)$.  Here $L_1$ is the cotangent line bundle at the first marked point of $\Mbar_{0,2}(X,d)$.  As with $J$, we write $\TT = \sum_{d} Q^d \TT_d$.

From the definitions of $J$ and $\TT$, we see that $J$-function is recovered as $\TT(1)$.  (The factor of $1/(1-q)$ in the $d>0$ terms of $J$ arises from the pushforward by the forgetful morphism $\Mbar_{0,2}(X,d) \to \Mbar_{0,1}(X,d)$, via the string equation in quantum K-theory; see \cite[\S4.4]{l}.)

\begin{remark}
\label{T-identity}
Note that $\TT|_{q=\infty} = \TT|_{Q=0} = \mathrm{id}$. In particular, the expansion of $\TT$ at $q=+\infty$ is of the form $\TT=\mathrm{id}+O(q^{-1})$.
\end{remark}

The coefficients $J_d$ and the operators $\TT_d$ can be computed by localization on the pointed graph space $\Gamma(X)_{n,d}$, and we mainly use this characterization.  
Consider the fixed component $\Gamma(X)_{n,d}^{(n,d)}$ which parametrizes stable maps in $\Mbar_{0,n}(\PP^1\times X, (1,d))$ whose source curve has a horizontal component of bi-degree $(1,0)$ and a vertical component of bi-degree $(0,d)$  attached to the horizontal component at $0$,  with all $n$ marked points lying on the vertical component. The key is an identification
\[
  \Gamma(X)_{n,d}^{(n,d)} \isom \Mbar_{0,n+1}(X,d)
\]
obtained by taking account of the node at $0$ where the vertical and horizontal components are attached.

Recall from \S\ref{subsec:moduli} that $\mathbb{C}^*$ acts on $\Gamma(X)_{n,d}$ via its action on $\mathbb{P}^1$, by $q\cdot [a,b]=[a,qb]$, fixing $0=[1,0]$ and $\infty=[0,1]$.  The normal bundle to the fixed component $\Gamma(X)_{n,d}^{(n,d)}$ has rank $2$, and decomposes into a trivial line bundle of character $q^{-1}$ (corresponding to moving the node away from $0$ along the horizontal curve), and a copy of the tangent line bundle $L^*_{n+1}$ on $\Mbar_{0,n+1}(X,d)$ with character $q^{-1}$ (corresponding to smoothing the node).  (See, e.g., \cite[p.~201]{gl}, \cite[Proof of Lemma~5.2]{bf1}, or \cite[\S1.3, \S3.3]{kontsevich}.)

Now the localization formula \eqref{e.fpformula} for the map $\mu_*\colon K^T_\circ(\Gamma (X)_d) \to K^T_\circ(\QM_d)$ says
\begin{equation}\label{e.Jd}
  \eem_{\QM_d^{(d)}}(\QM_d) = \mu^{(d)}_*\left(\frac{1}{\lambda_{-1}(N^*)} \right)
\end{equation}
where $\mu^{(d)}$ is the restriction of $\mu$ to the fixed component $\Gamma (X)_d^{(d)}$, $N$ is the normal bundle to this component, and $\lambda_{-1}(N^*) = 1 - N^* + \exterior^2 N^* - \cdots = (1-q)(1-qL)$.  Using the identifications $\QM_d^{(d)}\isom X$, $\Gamma (X)_d^{(d)} \isom \Mbar_{0,1}(X,d)$, and $\mu^{(d)} = \ev$, the right-hand side is exactly
\[
J_d = \ev_*\left( \frac{1}{(1-q)(1-qL)} \right).
\]
Similar reasoning identifies $\TT_d(\xi)$ as
\begin{equation}\label{e.Td}
  \frac{1}{1-q}T_d(\xi) = (\ev_1)_*\left( \frac{\ev_2^*\xi}{(1-q)(1-qL_1)} \right),\end{equation}
where 
we use the identification $\Gamma (X)_{1,d}^{(1,d)} \isom \Mbar_{0,2}(X,d)$.  See \cite[\S2.2 and \S4.2]{gl}.

Next we turn to the difference equations satisfied by $J$ and $\TT$.  The main theorems of \cite{gl}, \cite{bf1} say that $J$ is an eigenfunction of the finite-difference Toda operator \cite{e}, \cite{s}, \cite{ffjmm} when $X=G/B$ is a complete flag variety of type A, D, or E.  (A modification of $J$ satisfies the corresponding system in non-simply-laced types \cite{bf2}.)  We only need part of this structure, which holds for general $X$, suitably interpreted as in \cite{imt}.  To simplify the equations, we often write
\[
  \tilde{J} = P^{\log Q/\log q} J \quad \text{ and } \quad \tilde{\TT} = P^{\log Q/\log q} \TT,
\]
where $P^{\log Q/\log q}$ means $P_1^{\log Q_1/\log q}\cdots P_r^{\log Q_r/\log q}$.

Consider the {\em $q$-shift operator}  $\qshift{i}: Q_j \mapsto q^{\delta_{ij}}Q_j$ which induces an action on power series in $Q$.
The $\mathsf{D}_q$-module structure of quantum K-theory has the following form.

For a finite sequence $I$ consisting of integers $1\leq i\leq r$,
\begin{equation}
 \left(\prod_{i\in I}\qshift{i} \right)\tilde{J} = \tilde{\TT} \left(\prod_{i\in I} \Aa_r\qshift{i}(1)\right). \label{dq1} \end{equation}
where the $\Aa_i$ are certain operators in $\mathrm{End}_{R(T)}(K_T(X))\otimes \mathbb{Q}[q][[Q]]$ defined in \cite{imt}; see especially \cite[Proposition~2.10]{imt}.
This is essentially a commutation relation between the operators $\tilde\TT$ and $\qshift{i}$ which follows from \cite[Remark~2.11]{imt}.  
Note that \[\left(\prod_{i\in I}\qshift{i}\right)P^{\log Q/\log q} = \left(\prod_{i\in I} P_i\right)P^{\log Q/\log q} \left(\prod_{i\in I}\qshift{i}\right).\]  
Cancelling a factor of $P^{\log Q/\log q}$ and noting that $\prod_{i\in I}\qshift{i}$ operates by replacing $J_d$ with $q^{\sum_{i\in I}d_i} J_d$, we can rewrite Equation~\eqref{dq1} as
\begin{equation}\label{dq2}
\prod_{i\in I} P_i \left(\sum_{d\geq 0} q^{\sum_{i\in I}d_i} J_d Q^d \right)= \TT\left(\prod_{i\in I} \Aa_i\qshift{i}(1)\right).
\end{equation}

We can write $a_I:= \prod_{i\in I} \Aa_i\qshift{i}(1)$ as $a_I= \sum_{d\geq 0} a_I^{(d)} Q^d$ where each $a_I^{(d)}$ is polynomial in $q$ by  \cite[Proposition~2.10]{imt}. 
As noted in Remark~\ref{T-identity}, $\TT=\mathrm{id}+O(q^{-1})$, so we can rewrite Equation~\eqref{dq2}  as
\begin{equation}
\label{dq4}
 \prod_{i\in I}  P_i \left( 1 + \sum_{d>0} q^{d_i} J_d Q^d \right)=\TT(a_I)=a_I+\cdots =
      a_I^{(0)} + \sum_{d> 0} a_I^{(d)}  Q^d + \cdots,
\end{equation}
 where the omitted terms vanish at $q=\infty$.

Therefore the right-hand side of Equation~\eqref{dq4}---namely, the leading terms of $a_I$---can be studied from the asymptotics of the left-hand side as $q\to+\infty$, specifically, the $q^{\geq0}$ coefficients of $q^{\sum_{i\in I}d_i} J_d$.  We will see examples of how this works in Lemma \ref{l.Acom} and Proposition \ref{p.monomial} below.

\section{The zastava space and the $J$-function}

To bound the degrees $Q^d$ appearing in quantum products, our main tool will be a bound on the $q$-degree of the $J$-function and the operator $\TT$.  To obtain the required bound, we need some technical properties of a slice of the quasimap space, called the {\em zastava space}.  Definitions and detailed descriptions of this space can be found in \cite{bf1}, \cite[\S2]{bfg}, and \cite{bdf}.  (The last reference provides explicit coordinates.)  We briefly review the main properties of the zastava space, and study a particular desingularization of it by the (Kontsevich) graph space.

\subsection{Singularities of the zastava space}\label{subsec:zastava}

The zastava space $\Zas_d$ is an affine variety which can be thought of as a compactification of based maps $(\PP^1,\infty) \to (G/B,w_\circ)$.  It is defined as a locally closed subvariety of $\QM_d$, as follows.  Let $\QM_d^\circ$ be the open subset of quasimaps which have no ``defect'' at $\infty\in\PP^1$; i.e., the locus parametrizing maps defined in a neighborhood of $\infty$.  This comes with an evaluation morphism $\ev_\infty\colon \QM_d^\circ \to G/B$, and the zastava space is a fiber of this morphism: $\Zas_d =\ev_\infty^{-1}(w_\circ)$.  It has dimension $\mathrm{dim}\Zas_d = 2|d| = (2\rho,d)$.

A key property of the zastava space is that it stratifies into smaller such spaces.  Let $\Zas_d^\circ = \Zas_d \cap \Hom_d(\PP^1,G/B)$ be the open set of based maps.  Then
\[
  \Zas_d = \coprod_{0\leq d' \leq d} \Zas_{d'}^\circ \times \Sym^{d-d'}\AA^1,
\]
where for $e\in \check\Lambda_+$ the symmetric product $\Sym^e\AA^1$ is a space of ``colored divisors''.  Concretely, writing $e=e_1\check\alpha_1+\cdots+e_r\check\alpha_r$ with each $e_i\in \mathbb{Z}_{\geq0}$,
\[
  \Sym^e\AA^1 = \prod_{i=1}^r \Sym^{e_i}\AA^1.
\]
For any $d'\leq d$, let $\partial_{d'}\Zas_d\subseteq \Zas_d$ be the closure of the stratum $\Zas_{d-d'}^\circ\times\Sym^{d'}\AA^1$.  (See \cite[\S6]{bf1}.  By convention, let us declare $\partial_{d'}\Zas_d$ to be empty if $d' \not\leq d$.)  In particular, there are divisors $\partial_i\Zas_d := \partial_{\check\alpha_i}\Zas_d$.

We set
\[
  \Delta = \sum_{i=1}^r \partial_i\Zas_d
\]
and consider the pair $(\Zas_d,\Delta)$.  The strata of this pair can be described easily: for any $I\subseteq \{1,\ldots,r\}$, let
\[
 d_I = d - \sum_{i\in I} \check\alpha_i.
\]
Then
\[
  \Delta_I := \bigcap_{i\in I} \partial_i\Zas_d = \partial_{d_I}\Zas_d,
\]
understanding the RHS to be empty if $d_I\not\geq 0$.

Now consider the Kontsevich resolution of quasimaps by the graph space, $\Gamma (G/B)_d \to \QM_d$.  This restricts to an equivariant resolution of the zastava space, which we write as $\phi\colon \tilde\Zas_d \to \Zas_d$.  Let $\tilde\Delta$ be the proper transform of $\Delta$ under $\phi$; this is a simple normal crossings divisor.  Let $\tilde\omega$ and $\omega$ be the canonical sheaves of $\tilde\Zas_d$ and $\Zas_d$, respectively.  Our goal is to show the following:

\begin{proposition}\label{p.rational}
We have
\begin{align*}
  \phi_*\tilde\omega(\tilde\Delta) &= \omega(\Delta), \quad\text{ and }\\
 R^i\phi_*\tilde\omega(\tilde\Delta) &= 0 \qquad\text{ for }i>0.
\end{align*}
In particular, $\phi_*[\tilde\omega(\tilde\Delta)] = [\omega(\Delta)]$ as classes in $K_\circ^{\CC^*\times T}(\Zas_d)$.
\end{proposition}

\begin{proof}
We use the terminology and results of \cite[\S2.5]{kollar}.  In our context, this is the same as saying that $\phi\colon (\tilde\Zas_d,\tilde\Delta) \to (\Zas_d,\Delta)$ is a {\em rational resolution}.  By \cite[Proposition~2.84 and Theorem~2.87]{kollar}, it suffices to prove that the pair $(\Zas_d,\Delta)$ is {\em dlt} and the resolution $\phi\colon (\tilde\Zas_d,\tilde\Delta)\to (\Zas_d,\Delta)$ is {\em thrifty}.

The fact that $(\Zas_d,\Delta)$ is dlt is essentially proved in \cite{bf1,bf2}.  In fact, the proof of \cite[Proposition~5.2]{bf2} shows that $(\Zas_d,\Delta)$ is a klt pair, since $\omega(\Delta)$ is Cartier (in fact, trivial) and the relative log canonical divisor of the resolution $\phi$ has nonnegative coefficients.  Since klt implies dlt, this suffices (see \cite[Definition~2.8]{kollar}).

The notion of a thrifty resolution $f\colon (Y,D_Y) \to (W,D)$ is defined in \cite[Definition~2.79]{kollar}: this means that $W$ is normal, $D$ is a reduced divisor, $D_Y$ is the proper transform of $D$ and has simple normal crossings, $f$ is an isomorphism over the generic point of every stratum of the snc locus $\mathrm{snc}(W,D)$, and $f$ is an isomorphism at the generic point of every stratum of $(Y,D_Y)$.

The fact that $\phi\colon(\tilde\Zas_d,\tilde\Delta) \to (\Zas_d,\Delta)$ satisfies these conditions is straightforward.  To check it, we review the description of $\phi$, considering its values on strata.  The component $\tilde\partial_i$ is the proper transform of $\partial_i = \partial_i\Zas_d \subseteq \Zas_d$; a general point parametrizes stable maps whose source curve has a vertical component of degree $\check\alpha_i$, attached to a horizontal component of degree $d-\check\alpha_i$ at some point $x\neq \infty$.  By remembering the map $f$ from the horizontal component and the point $x$ where the vertical component is attached, this maps to $(f,x)\in \Zas^{\circ}_{d-\check\alpha_1} \times \AA^1$.

Similarly, suppose $I=\{i_1,\ldots,i_k\}$ indexes a stratum.  A general point of $\tilde\Delta_I = \cap_{i\in I} \tilde\partial_i$ consists of maps from a source curve with vertical components of degrees $\check\alpha_i$, one for each $i\in I$, attached to a horizontal component of degree $d'=d-\sum_{i\in I} \check\alpha_i$ at distinct points $x_{i_1},\ldots,x_{i_k}$.  This maps to $(f, x_{i_1},\ldots,x_{i_k})\in \Zas^\circ_{d'}\times (\AA^1)^k$, as before.  Since the map $\tilde\Zas_{d'}\to \Zas_d$ is birational, so is the map of strata $\tilde\Delta_I \to \Delta_I$.

Finally, no subvariety of $\tilde\Zas_d$ other than $\tilde\Delta_I$ maps onto the stratum $\Delta_I$.  Indeed, $\Delta_I$ is the closure of $\Zas_{d'}\times (\AA^1)^k$, with notation as in the previous paragraph, so a general point will have $k$ distinct coordinates $x_{i_1},\ldots,x_{i_k}$ for the $(\AA^1)^k$ factor.  The only preimage under $\phi$ of such a point is a map $(f,x_{i_1},\ldots,x_{i_k})$ as described above.\footnote{There are other subvarieties of $\tilde\Zas_d$ mapping into $\Delta_I$, but not dominantly.  For instance, there is a divisor $D_{\check\alpha_1+\check\alpha_2} \subseteq\tilde\Zas_d$ where the source curve has a vertical component of degree $\check\alpha_1+\check\alpha_2$ attached at a point $x$ to a horizontal component of degree $d-\check\alpha_1-\check\alpha_2$.  This maps to $\partial_1\cap \partial_2$, but in the stratum $\Zas^\circ_{d-\check\alpha_1-\check\alpha_2}\times (\AA^1)^2$, the image only contains points in the diagonal $\AA^1 = \{(x,x)\} \subseteq (\AA^1)^2$.}
\end{proof}

\subsection{Asymptotics of the $J$-function}

A key ingredient in our approach to finiteness is a bound on the growth of the coefficients $J_d$, and more generally $\TT_d$, when considered as rational functions of $q$.  Here we consider $G/B$; the extension to general $G/P$ will be addressed later.

Given any $d\in\check\Lambda_+$, define
\begin{equation}\label{kd}
  m_{d} :=  r(d) + \frac{(d,d)}{2},
\end{equation}
where $r(d)$ is the number of $i$ such that $d_i>0$.

Writing $J = \sum_d Q^d J_d$, each $J_d$ is a rational function in $q$, with coefficients in $K_T(G/B)$.  As $q\to\infty$, then, $J_d$ tends to $c_d\,q^{-\nu_{d}}$, for some element $c_{d}\in K_T(G/B)$ and some integer $\nu_{d}$.

\begin{lemma}\label{l.bound}
We have $\nu_{d} \geq m_{d}$.
\end{lemma}

\begin{proof}
Because $\mathbb{C}^*$ acts trivially on $G/B$, it is enough to compute the asymptotics of the restriction of $J_d$ to any fixed point in $(G/B)^T$; we choose the point $w_\circ$, corresponding to the longest element of the Weyl group.

By Equation~\eqref{e.Jd}, the restriction $J_d|_{w_\circ}$ is equal to the contribution from the fixed point $(d,w_\circ)\in \QM_d^{\mathbb{C}^*\times T}$ appearing in the localization formula for $\chi(\QM_d,\,\OO)$.  The localization formula \eqref{e.fpformula}, applied to the map $\QM_d \to \mathrm{pt}$, says
\[
  \chi(\QM_d,\,\OO) = \sum_{(d^+,w)} \eem_{(d^+,w)}(\QM_d).
\]
So we only need to compute the equivariant multiplicity, or more specifically, its degree as a rational function in $q$.

We may reduce to the zastava space $\Zas_d$; from its description as the fiber over $w_\circ\in G/B$ of the evaluation map $\ev_\infty\colon \QM_d^{\circ} \to G/B$, we see that
\[
\eem_{(d,w_\circ)}(\QM_d) = \left(\prod \frac{1}{1-\ee^{-\alpha}}\right) \cdot \eem_{0}(\Zas_d),
\]
where the product is over positive roots $\alpha$.  In particular, the contribution of $q$ to $\eem_{(d,w_\circ)}(\QM_d)$ comes from $\eem_0(\Zas_d)$, so it is enough to compute the latter.

Let us write
\[
  \eem_0(\Zas_d) = \frac{R(q)}{S(q)}
\]
as a rational function in $q$.  We wish to show
\begin{equation}\label{bd1}
 \deg(R) - \deg(S) \leq -m_{d} = -r(d)-\frac{(d,d)}{2},
\end{equation}
or in other words, the order of the rational function is $\mathrm{ord}_\infty(\eem_0(\Zas_d)) \geq m_{d}$.  This will give the asserted bound.

Using the notation of Proposition~\ref{p.rational}, recall $\omega=\omega_{\Zas_d}$ is the canonical sheaf, and $\Delta\subseteq\Zas_d$ is the boundary divisor.  By the proof of \cite[Proposition~5.2]{bf2}, $\omega(\Delta)$ is a trivial line bundle, with $q$-weight $(d,d)/2=m_d-r(d)$, so
\begin{equation}\label{bd2}
  \mathrm{ch}(\omega(\Delta))  = q^{m_{d}-r(d)}\,\eem_0(\Zas_d).
\end{equation}
We will show that the rational function $\mathrm{ch}(\omega(\Delta))$ has $\mathrm{ord}_\infty(\mathrm{ch}(\omega(\Delta)))\geq r(d)$, which proves Equation~\eqref{bd1} after dividing by $q^{m_{d}-r(d)}$.

To see this, we will compute $\mathrm{ch}(\omega(\Delta))$ by localization, using the Kontsevich resolution and the identity $[\omega(\Delta)]=\phi_*[\tilde\omega(\tilde\Delta)]$ from Proposition~\ref{p.rational}.  Recalling the descriptions of the $\CC^*$-fixed components of $\Gamma (G/B)_d$, one sees that $\tilde\Zas_d$ has a unique fixed component, namely
\[
  \mathcal{F}=\tilde\Zas^{\CC^*}_d=\Gamma (G/B)_d^{(d)}\cap \tilde\Zas_d.
\]
A general point parametrizes based maps where the source curve consists of a horizontal component of degree 0 (mapping to $w_\circ \in G/B$) with a vertical component of degree $d$, attached to the horizontal component at the fixed point $0$.

Now we have
\begin{equation}
\mathrm{ch}(\omega(\Delta)) =  \eem_0(\Zas_d)\cdot [\omega(\Delta)]|_0 =  \phi_*\left( \frac{\tilde\omega(\tilde\Delta)|_\mathcal{F}}{\lambda_{-1}( N^*_{\mathcal{F}/\tilde\Zas_d})} \right).
\end{equation} 
Taking $q$-graded characters, the fraction in the right-hand side has order $r(d)$ at $q=\infty$.  Indeed, the nontrivial characters appearing in $\tilde\omega|_\mathcal{F}$ are precisely those appearing as normal characters in $N_{\mathcal{F}/\tilde\Zas_d}$.  (The tangential directions along $\mathcal{F}$ have trivial character, since $\mathcal{F}$ is fixed.)  Each irreducible component of the divisor $\tilde\Delta$ contributes $q^{-1}$, by the proof of \cite[Lemma~5.2]{bf1}, and there are $r(d)$ such components.  Finally, after pushing forward by $\phi$, we see that the order at $\infty$ of the right-hand side is at least $r(d)$. (Some terms may vanish in the pushforward, so inequality is possible.)
\end{proof}

In the case where $G$ is simply laced---i.e., of type A, D, or E---a similar (but simpler) argument produces a stronger bound.  Let $k_{d} :=  (\rho,d) + \frac{(d,d)}{2}$.

\renewcommand{\thetheorem}{\ref{l.bound}$^+$}
\begin{lemma}\label{l.boundADE}
When $G$ is simply laced, we have $\nu_{d} \geq k_{d}$.
\end{lemma}
\addtocounter{theorem}{-1}
\renewcommand{\thetheorem}{\arabic{theorem}}

\begin{proof}
The argument is exactly as before, with the following changes.  First, we have that $\omega$ itself is a trivial line bundle with character $q^{(\rho,d)+(d,d)/2}$, as in the proof of \cite[Lemma~5.2]{bf1}, so that
\[
  \mathrm{ch}(\omega) = q^{k_{d}}\,\eem_0(\Zas_d).
\]
Next, we have $\phi_*[\tilde\omega] = [\omega]$ using the fact that $\Zas_d$ has rational singularities \cite[Proposition~5.1]{bf1}.  Finally, the fraction
\[
  \frac{\tilde\omega|_\mathcal{F}}{\lambda_{-1}(N^*_{\mathcal{F}/\tilde\Zas_d})}
\]
has order $0$ at infinity, so pushing forward by $\phi$ shows that $\mathrm{ord}_\infty(\mathrm{ch}(\omega))\geq 0$.  Dividing by $q^{k_{d}}$ yields the bound.
\end{proof}

\begin{remark*}
In type A, the exponent is
\[
 k_{d} = d_1+\cdots+ d_r  + \sum_{i=1}^{r+1}\frac{(d_i-d_{i-1})^2}{2} ,
\]
where $d_0=d_{r+1}=0$, which agrees with \cite[Eq.~(7)]{gl}.
\end{remark*}

\begin{remark*}
For any smooth projective variety $X$, using the characterization of $J_d$ as
\[
J_d = \ev_*\left( \frac{1}{(1-q)(1-qL)} \right),
\]
where $\ev\colon \Mbar_{0,1}(X,d) \to X$ is the evaluation, one can interpret $\nu_d$ as the minimal integer $\geq 2$ such that $\ev_*(L^{-\nu_d+1})\neq 0$ in $K_T(X)$.  Indeed, one expands this pushforward in powers of $q^{-1}$ as
\[
 q^{-2} (1 + q^{-1} + q^{-2} + \cdots ) \ev_*\left( L^{-1}( 1 + q^{-1} L^{-1} + q^{-2} L^{-2} + \cdots ) \right).
\]
A similar characterization of the order of $\TT_d$ at $q=\infty$ will be useful below.
\end{remark*}

\subsection{Comparison between the Borel and parabolic cases}\label{subsec:parabolic}

We compare the vanishing orders (at $q=\infty$) of $\TT$ for $G/B$ and $G/P$.  Our main tool is a construction due to Woodward, in the course of his proof of the Peterson-Woodward comparison formula relating quantum cohomology of $G/P$ to that of $G/B$ \cite{w}.  

Given any $d_P\geq 0$ in $\check\Lambda^P$, the Peterson-Woodward formula produces another parabolic $P'$, with $P\supseteq P'\supseteq B$, together with canonical lifts $d_{P'} \in \check\Lambda^{P'}_+$ and $d_B \in \check\Lambda_+$ of $d_P$.  There are natural morphisms
\begin{align*}
h_{P'/B}\colon \Gamma (G/B)_{n,d_B} & \to \Gamma (G/P')_{n,d_{P'}} \times_{G/P'} G/B \\
\intertext{and}
h_{P/P'}\colon \Gamma (G/P')_{n,d_{P'}} & \to \Gamma (G/P)_{n,d_P},
\end{align*}
where $\Gamma (G/B)_{n,d_B}  \to \Gamma (G/P')_{n,d_{P'}}$ and $\Gamma (G/P')_{n,d_{P'}}  \to \Gamma (G/P)_{n,d_P}$ come from functoriality of the Kontsevich space, and $\Gamma (G/B)_{n,d_B} \to G/B$ and $\Gamma (G/P')_{n,d_{P'}} \to G/P'$ are given by evaluation at $0\in \PP^1$.  (This makes sense, since any source curve in the graph space has a distinguished component together with a fixed isomorphism to $\PP^1$.)

Woodward shows that these morphisms are birational.  More precisely, \cite[Theorem~3]{w} asserts that the corresponding maps between $\Hom$ spaces are birational, and these are dense open sets in our graph spaces.

Explicit formulas for $d_B$ and $P'$ can be found in \cite[Remark~10.17]{ls}, but for our purposes it is enough to know that $d_B$ and $d_{P'}$ map to $d_P$ under the canonical projection, and that the above birational morphisms exist.

Consider $d_P\geq 0$ in $\check\Lambda^P$, and let us define $\nu_{d_P}$ as for the $G/B$ case: it is the exponent so that $J_{d_P}$ tends to $c_{d_P}\,q^{-\nu_{d_{P}}}$ as $q\to\infty$, for some $c_{d_P}\in K_T(G/P)$.  In other words, $\nu_{d_P} = \mathrm{ord}_{\infty}(J_{d_P})$.

In addition to the Peterson-Woodward lift $d_B$ of a degree $d_P\in \check\Lambda^P_+$, there is another canonical lift, which we call the {\em minimal} lift $d_B^{\rm min} \in \check\Lambda_+$.  This is (unique) smallest effective lift of $d_P$.  Explicitly, write $d_P = \sum c_i \bar{\alpha}_i$, where the sum is over $i\not\in I_P$, each $c_i\geq 0$, and $\bar{\alpha}_i$ is the image of $\check\alpha_i$ in $\check\Lambda^P$.  Then $d_B^{\rm min} = \sum c_i \check\alpha_i$.

Here is the main lemma of this section.
\begin{lemma}\label{l.boundP}
For any $\xi\in K_T(G/P)$, we have
\[
 \ord_{q=\infty}\TT_{d_P}(\xi) \geq \min_{d_B^{\rm min} \leq d_B^+ \leq d_B} \{ \ord_{q=\infty}\TT_{d_B^{+}}(\pi^*\xi) \},
\] 
where $\pi\colon G/B \to G/P$ is the projection.  
In particular, taking $\xi=1$, we have  $\nu_{d_P} \geq \min_{d_B^{\rm min} \leq d_B^+ \leq d_B}\{ m_{d_B^+} \}$.
\end{lemma}

\begin{proof}
When $\xi=1$, the displayed inequality is precisely $\nu_{d_P}\geq  \min_{d_B^{\rm min} \leq d_B^+ \leq d_B}  \{ \nu_{d_B^{+}}\}$, so the second statement follows from Lemma~\ref{l.bound}.

To verify $\ord_{q=\infty}\TT_{d_P}(\xi) \geq \ord_{q=\infty}\TT_{d_B^{\rm min}}(\pi^*\xi)$, we use the characterization $\TT_d(\xi) = (\ev_1)_*\left( \frac{\ev_2^*\xi}{1-qL_1} \right)$ from Equation~\eqref{e.Td}, where $\ev_i\colon \Mbar_{0,2}(X,d) \to X$ are the evaluation maps.  Let
\[
 h \colon \Gamma(G/B)_{n,d_B} \to \Gamma(G/P)_{n,d_P}
\]
be the composition of $h_{P'/B}$, the projection on the first factor, and $h_{P/P'}$.  The $\CC^*$-fixed loci of $\Gamma(G/B)_{n,d_B}$ which map to the fixed component $\Gamma(G/P)_{1,d_P}^{(1,d_P)}$ are the components $\Gamma(G/B)_{1,d_B}^{(1,d_B^+)}$ such that $d_B^{\rm min} \leq d_B^+ \leq d_B$.  Recall the identifications of fixed loci
\begin{align*}
 \Gamma(G/B)_{1,d_B}^{(1,d^+_B)} &\isom  \Mbar_{0,2}(G/B,d^+_B) \times_{G/B} \Mbar_{0,1}(G/B,d_B^-) \quad \text{ and } \\
 \Gamma(G/P)_{1,d_P}^{(1,d_P)} &\isom  \Mbar_{0,2}(G/P,d_P),
\end{align*}
where in the fiber product both maps to $G/B$ are by $\ev_1$.  We have commutative diagrams
\[
\begin{tikzcd}
G/B  \ar{d}{\pi} & \Mbar_{0,2}(G/B,d_B^+) \times_{G/B} \Mbar_{0,1}(G/B,d_B^-) \ar{d}{\bar{h}^{d_B^+}} \ar{l}{\ev_1} \ar[hook]{r}{\iota}  &  \Gamma(G/B)_{1,d_B} \ar{r}{\ev} \ar{d}{h} & G/B \ar{d}{\pi} \\
G/P & \Mbar_{0,2}(G/P,d_P)  \ar{l}{\ev_1}   \ar[hook]{r}{\iota}  &  \Gamma(G/P)_{1,d_P}    \ar{r}{\ev} & G/P
\end{tikzcd}
\]
for each such $d_B^+$, where $d_B=d_B^+ + d_B^-$.  
In the bottom row, the composition $\ev\circ \iota$ is equal to $\ev_2\colon \Mbar_{0,2}(G/P,d_P)\to G/P$, and similarly in the top row (when one also composes with the projection on the first factor).  Since $h$ is the composition of birational morphisms between varieties with rational singularities and a smooth projection with rational fibers, we have $h_*h^*(z) = z$ for any $z\in K_T(\Gamma(G/P)_{1,d_P}  )$.  Furthermore, by the localization formula (\ref{e.fpformula}) applied to $\bar{h}$, for any $\alpha\in K_T(\Gamma(G/B)_{1,d_B})$ we have
\begin{align*}
  \frac{\iota^*h_*(\alpha)}{(1-q)(1-qL_1^P)} &=  \bar{h}^{d_B}_*\left( \frac{\iota^*\alpha}{(1-q)(1-qL_1)} \right) \\
  &\quad + \sum_{d_B^{\rm min} \leq d_B^+ < d_B} \bar{h}^{d_B^+}_*\left( \frac{\iota^*\alpha}{(1-q)(1-qL_1)(1-q^{-1})(1-q^{-1}L')} \right) .
\end{align*}
Here $L_1^P$ is the cotangent line bundle at the first marked point of $\Mbar_{0,2}(G/P,d_P)$, and $L_1$ and $L'$ are the pullbacks of cotangent line bundles on $\Mbar_{0,2}(G/B,d^+_B)$ and $\Mbar_{0,1}(G/B,d^-_B)$, respectively.  (The denominators are the K-theoretic top Chern classes of the normal bundles to the respective fixed loci, see e.g. \cite[\S2.6]{imt}.)

Now we set $\alpha = \ev^*\pi^*\xi = h^*\ev^*\xi$ in the above equation and apply $(\ev_1)_*$ to both sides.  On the left-hand side, we obtain
\[
  (\ev_1)_*\left( \frac{\iota^*h_*h^*\ev^*\xi}{(1-q)(1-q L^P_1)} \right) = (\ev_1)_*\left( \frac{\ev_2^*\xi}{(1-q)(1-q L^P_1)} \right) = \frac{1}{1-q}T_{d_P}(\xi).
\]
For the first term on the right-hand side, we compute
\begin{align*}
   (\ev_1)_*\bar{h}^{d_B}_*\left( \frac{\iota^*h^*\ev^*\xi}{(1-q)(1-q L_1)} \right)
   &= \pi_*(\ev_1)_*\left( \frac{\iota^*\ev^*\pi^*\xi}{(1-q)(1-q L_1)} \right) \\
    &= \pi_*(\ev_1)_*\left( \frac{\ev_2^*\pi^*\xi}{(1-q)(1-q L_1)} \right) \\
   & = \frac{1}{1-q}\pi_*\TT_{d_B}(\pi^*\xi),
\end{align*}
so this term vanishes to order at least $\ord_{q=\infty}\TT_{d_B}(\pi^*\xi)$.

The other terms are similar.  Writing $\mathrm{pr}\colon \Mbar_{0,2}(G/B,d_B^+) \times_{G/B} \Mbar_{0,1}(G/B,d_B^-) \to \Mbar_{0,1}(G/B,d_B^-)$ for the second projection, we have
\begin{align*}
  & (\ev_1)_*\bar{h}^{d_B^+}_*\left( \frac{\iota^*h^*\ev^*\xi}{(1-q)(1-q L_1)(1-q^{-1})(1-q^{-1}L')} \right) \\
   &= \frac{1}{(1-q)(1-q^{-1})} \pi_*(\ev_1)_*\left( \mathrm{pr}_*\left(\frac{\iota^*\ev^*\pi^*\xi}{1-q L_1}\right) \cdot \frac{1}{1-q^{-1}L'}  \right) \\
    &= \frac{q^{-2}}{(1-q^{-1})^2} \pi_*(\ev_1)_*\left(\mathrm{pr}_*\left( \ev_2^*\pi^*\xi \cdot L_1^{-1}(1+q^{-1}L_1^{-1} + q^{-2}L_1^{-2} + \cdots) \right) \cdot \frac{1}{1-q^{-1}L'} \right).
\end{align*}
The factor
\[
  \frac{q^{-2}}{(1-q^{-1})^2}\mathrm{pr}_*\left( \ev_2^*\pi^*\xi \cdot L_1^{-1}(1+q^{-1}L_1^{-1} + q^{-2}L_1^{-2} + \cdots) \right)
\]
has vanishing order equal to $\ord_{q=\infty}\TT_{d^+_B}(\pi^*\xi)$, since $\mathrm{pr}$ is a flat pullback of the evaluation map $\ev_1\colon \Mbar_{0,2}(G/B,d_B^+) \to G/B$ which computes $\TT_{d_B^+}$.  So the whole term vanishes at least to order $\ord_{q=\infty}\TT_{d^+_B}(\pi^*\xi)$.  Our claim follows.
\end{proof}

When $G$ is simply laced, the same argument produces a sharper bound:
\renewcommand{\thetheorem}{\ref{l.boundP}$^+$}
\begin{lemma}\label{l.boundPADE}
If $G$ is simply laced, we have $\nu_{d_P} \geq \min_{d_B^{\rm min} \leq d_B^+ \leq d_B}\{ k_{d_B^+} \}$. \qed
\end{lemma}
\addtocounter{theorem}{-1}
\renewcommand{\thetheorem}{\arabic{theorem}}

\begin{remark*}
For degrees $d_P$ such that $d_B=d_B^{\rm min}$, the same argument shows that
\[
  \TT_{d_P}(\xi) = \pi_*\TT_{d_B}(\pi^*\xi),
\]
since in this case there is only one term in the localization formula.
\end{remark*}

\section{The operator $\Acom{i}$}\label{sec:Acom}

For a partial flag variety $G/P$ and a degree $d=d_P$, we write $\hat{d}$ for an associated degree on $G/B$ which lies in the interval between $d_B^{\rm min}$ and $d_B$, and achieves the minimum of $m_{d_B^+}$ among degrees $d_B^+$ in this interval, as in as in \S\ref{subsec:parabolic}.  That is,
\[
m_{\hat{d}} = \min_{d_B^{\rm min} \leq d_B^+ \leq d_B} \{ m_{d_B^+} \},
\]
and by Lemma~\ref{l.boundP}, we have $\nu_{d} \geq m_{\hat{d}}$.

As discussed in \S\ref{subsec:J},  certain operators $\Aa_i\in \mathrm{End}_{R(T)}(K_T(G/P))\otimes \mathbb{Q}[q][[Q]]$, defined and studied in \cite{imt},  give the $\mathsf{D}_q$-module structure of quantum K-theory.

 Setting $q=1$ in $\Aa_i$ produces operators $\Acom{i} := \Aa_i|_{q=1}\in \mathrm{End}(K_T(G/P))\otimes \mathbb{Q}[[Q]]$.  By Equation~\eqref{dq2} and \cite[Proposition~2.12]{imt}, we have
\begin{equation}\label{dq3}
\prod_{i\in I} \Acom{i} (1) =  a_I|_{q=1}
\end{equation}

\begin{lemma}\label{l.Acom}
The operator $\Acom{i}$ is the operator of the (small) quantum product by $P_i$.
\end{lemma}

\
\begin{proof}
It suffices to show that $\Acom{i}(1)=P_i$. 
By \cite[Proposition~2.10]{imt}, the operators $\Acom{i}$ act as the (small) quantum product: we have
\begin{equation}\label{acom1}
  \Acom{i}(\Phi) = \left( P_i + \sum_{d>0} c_{d,i} Q^d \right)\star \Phi,
\end{equation}
for some $c_{d,i}\in K_T(G/P)$.  We will prove that $c_{d,i}=0$ for all $d>0$.

Writing $a=\Aa_i\qshift{i}(1)$ and applying Equation~\eqref{dq4}, we obtain
\begin{equation}
 P_i \left( 1 + \sum_{d>0} q^{d_i} J_d Q^d \right)  = \TT(a) = a^{(0)} 
 +\sum_{d>0} a^{(d)} Q^d+\cdots, \label{e.Acom}
 \end{equation}
 where the omitted terms vanish at $q=\infty$.

As in the discussion after Equation~\eqref{dq4}, we compute $\Acom{i}(1) = a|_{q=1}$ by studying the asymptotics of the expansion of the left-hand side of \eqref{e.Acom} at $q=\infty$.  

To prove the lemma, we wish to show $a^{(d)}=0$ for $d>0$.  For this, since we know that $a^{(d)}$ is polynomial in $q$, it suffices to show that there are no $q^{\geq 0}$ coefficients of $Q^d$ on the left-hand side of \eqref{e.Acom}.

Suppose a $d>0$ term contributes to the $q^{\geq 0}$ coefficients---that is, suppose $q^{d_i}J_d$ has non-positive order at $q=\infty$.  This means that $d_i\geq \nu_d$.  Noting that $\hat{d}_i=d_i$ since $\hat{d}=d_B$ is a lift of $d=d_P$, Lemma~\ref{l.boundP} 
gives 
\begin{align}
  0 \leq d_i - \nu_d &\leq d_i - m_{\hat{d}} = \hat{d}_i-m_{\hat{d}} \label{aicom-bound}. 
\end{align}
By the Lemma in Appendix \ref{app:ineq}, when $G$ contains no simple factors of type E$_8$, the right-most term is strictly negative when $d>0$, giving a contradiction. For the E$_8$ case we have the stronger bound of Lemma~\ref{l.boundPADE} which applies to all simply laced types (see Lemma \ref{l.AcomADE} below).  Therefore, no such $d>0$ terms arise, and the lemma is proved. 

\end{proof}

In the simply-laced case, we can say more.

\renewcommand{\thetheorem}{\ref{l.Acom}$^+$}
\begin{lemma}\label{l.AcomADE}
If $G$ is simply laced, then for {\em distinct} $i_1,\ldots,i_l\in \{1,\ldots,r\}$, we have $P_{i_1}\star \cdots\star P_{i_l}=\prod_{k=1}^lP_{i_k}$.  That is, for these elements, the quantum and classical product are the same. \qed
\end{lemma}
\addtocounter{theorem}{-1}
\renewcommand{\thetheorem}{\arabic{theorem}}
\begin{proof}

It suffices to show that for distinct $i_1,\ldots,i_l\in \{1,\ldots,r\}$, we have
\[
\left(\prod_{k=1}^l \qshift{{i_k}} \right) \tilde{J}= \tilde{\TT}\left(\prod_{k=1}^lP_{i_k}\right).
\]
This follows from the same argument as in the proof of Lemma \ref{l.Acom}. Indeed, the inequality in Equation~\eqref{aicom-bound} can be replaced by
\begin{align*}
  0 \leq \sum_{k=1}^l d_{i_k} - \nu_d &\leq  \sum_{k=1}^l d_{i_k} - k_{\hat{d}}\\
    & = -\left(\rho-\sum_{k=1}^l\varpi_{i_k}, \hat{d}\right) - \frac{(\hat{d},\hat{d})}{2}.
\end{align*}
The quantity $\left(\rho-\sum \varpi_{i_k}, \hat{d}\right)$ is nonnegative, and $\frac{(\hat{d},\hat{d})}{2}$ is strictly positive for $d\neq 0$, since $(\;,\;)$ is an inner product.  This contradicts the inequality, so no term with $d>0$ occurs.
\end{proof}

\section{Asymptotics of the fundamental solution $\TT$}

We would like to establish a generalization of Lemma \ref{l.bound} (and Lemma \ref{l.boundADE} in simply-laced cases) to $\mathsf{T}_d$ by further exploring the properties of the zastava spaces.  Alternatively, one may hope to derive such a generalization with the help of reconstruction theorems \cite{imt}, \cite{lp}. The subtleties involved in either approach present formidable technical challenges.

We proceed differently.  Lemmas \ref{l.bound} and \ref{l.boundADE} imply that $J_d$ satisfies a {\em quadratic growth condition} in the sense introduced in Appendix \ref{app:iritani} by H.~Iritani.  More precisely, for any smooth projective variety $X$, a linear operator $\mathsf{T}=\sum \mathsf{T}_d Q^d$ on $K_T(X)$ {\bf satisfies the quadratic growth condition} if there is a positive-definite inner product $(\;,\;)$ on $H_2(X)$, a linear functional $m$ on $H_2(X)$, and a real constant $c$ such that
\[
  \ord_{q=\infty} \mathsf{T}_d \geq \frac{(d,d)}{2} + m(d) + c
\]
for all $d\in H_2(X)$.  
In the appendix, Iritani proves that the quadratic growth condition on the fundamental solution $\TT$ is equivalent to the shift operators $\mathsf{A}_i$ being polynomials in the Novikov variables $Q$.

According to Kato \cite[Corollary~3.3]{k1} (which uses our Lemma \ref{l.Acom}), for $G/B$ the shift operators $\mathsf{A}_i$ are polynomials in Novikov variables $Q$.  Applying Iritani's result (the Proposition of Appendix~\ref{app:iritani}), we obtain:

\begin{lemma}\label{l.T-bound}
For $G/P$, the fundamental solution $\mathsf{T}$ satisfies the quadratic growth condition.
\end{lemma}

\begin{proof}
By Iritani's Proposition and Kato's finiteness result for $G/B$ \cite[Corollary~3.3]{k1}, the operator $\mathsf{T}$ for $G/B$ satisfies the quadratic growth condition.  Using the bounds of Lemma~\ref{l.boundP}, the quadratic growth condition for $G/P$ follows.
\end{proof}

Applying the Proposition of Appendix~\ref{app:iritani} again, it follows that the shift operators $\mathsf{A}_i$ for $G/P$ are also polynomials in $Q$.  We give a direct argument for this last implication in the next section.

Arguing as in the proof of \cite[Lemma 3.3]{imt}, we have the following lemma, which will be used in Section~\ref{sec:finite}.

\begin{lemma}\label{l:T-infty} Consider $\UU\in K_T(G/P)[q][[Q]]$ such that
 $\TT(\UU)=0$ at $q=\infty$. Then $\TT(\UU)=0$.
\end{lemma}
\begin{proof}
Write $\MM:=\TT(\UU)$. Expanding $\MM=\sum_d \MM_d Q^d$, $\TT=\sum_d \TT_d Q^d$, and $\UU=\sum_d \UU_d Q^d$ as series in $Q$, we will show $\MM=0$ by induction with respect to a partial order on effective curve classes $d\in\check\Lambda_+$.  In fact, we will show $\UU_d=0$ for all $d$.

As rational functions in $q$, the coefficients $\TT_d$ and $\UU_d$ have the following properties: $\TT_0=\mathrm{id}$; $\TT_d$ has poles only at roots of unity, is regular at $q=0$ and $q=\infty$, and vanishes at $q=\infty$ for $d>0$; and $\UU_d$ is a polynomial in $q$.  Since $\TT_0=\mathrm{id}$, it follows that $\UU_0=0$.

The product formula expands to give
\[
 \MM_d=\UU_d+ \sum_{\substack{d'+d''=d \\ d', d''> 0}}\TT_{d'}\UU_{d''},
\]
using $\TT_d(\UU_0)=\TT_d(0)=0$.  
By induction, the indexed sum is zero (since all lower terms $\UU_{d''}=0$), i.e., $\MM_d=\UU_d$.  Since $\MM_d$ vanishes at $q=\infty$ for all $d$, but $\UU_d$ is a polynomial in $q$, it follows that $\UU_d=0$ and $\MM=0$.
\end{proof}

\section{Finiteness}\label{sec:finite}

We will deduce our main finiteness theorem from the following statement for products of the line bundle classes $P_i$.  This argument originally appeared in our preprint \cite[Proposition 5]{act}.

\begin{proposition}\label{p.monomial}
For any indices $i_1,\ldots,i_l$, the (small) quantum product $P_{i_1}\star \cdots\star P_{i_l}$ is a finite linear combination of elements of $K_T(G/P)$ whose coefficients are polynomials in $Q_1,\ldots,Q_r$.
\end{proposition}

The statement is similar to the ``only if'' direction of Iritani's Proposition in Appendix~\ref{app:iritani}, but phrased differently.  In our context, because of Lemma~~\ref{l.Acom}, polynomiality of $\Acom{i}$ is equivalent to that of quantum multiplication by $P_i$.

In proving Proposition~\ref{p.monomial}, we will extend scalars from $R(T)$ to $F(T)$, and choose an $F(T)$-basis $\Phi_w = P^{\lambda(w)}$ of line bundles, for some $\lambda(w)\in \Lambda$.  (By Lemma~\ref{l.generate}, $F(T)\otimes_{R(T)} K_T(G/P)$ is generated by line bundles over $F(T)$, so such a monomial basis exists.)  This extension of scalars is harmless, for the following reason.  A priori, we know the quantum product $P_{i_1} \star \cdots \star P_{i_l}$ lies in $K_T(G/P)[[Q]]$.  The argument we give below shows that it lies in $(F(T)\otimes_{R(T)} K_T(G/P))[Q]$.  This proves the claim, because the intersection of the submodules $K_T(G/P)[[Q]]$ and $(F(T)\otimes_{R(T)} K_T(G/P))[Q]$ inside $(F(T)\otimes_{R(T)} K_T(G/P))[[Q]]$ is precisely $K_T(G/P)[Q]$.

\begin{proof}
From Equation~\eqref{dq1}, for $I=(i_1,\ldots,i_l)$ we have
\begin{equation} \label{e.monomial3}
  \left(P^{\log Q/\log q}\right)^{-1}\prod_{k=1}^l \qshift{{i_k}}\tilde{J} = \TT(a_I),
\end{equation}
where $a_I:= \prod_{i\in I} \Aa_i\qshift{i}(1)\in  F(T)[q][[Q]]$ by \cite[Proposition~2.10]{imt}. This can be rewritten as in Equation~\eqref{dq2}, as
\begin{equation}\label{e.monomial2}
 \prod_{k=1}^l P_{i_k} \left(\sum_{d\geq 0} q^{\sum_{k=1}^l d_{i_k}} J_d Q^d \right) =\TT(a_I)=a_I^{(0)}+\sum_{d>0} a_I^{(d)}Q^d +\ldots
\end{equation}
where the omitted terms vanish at $q=\infty$ (since $\TT=\mathrm{id} + O(q^{-1})$).

By Lemma~\ref{l.Acom}, the operator $\Acom{i}$ is the operator of quantum multiplication by $P_i$. Along with Equation~\eqref{dq3} for $I=(i_1,\ldots,i_l)$, we obtain
\begin{equation}\label{P-product}
 P_{i_1}\star \cdots\star P_{i_l}=\prod_{k=1}^l \Acom{{i_k}}(1) =a_I|_{q=1}.
 \end{equation}
Our goal is to show that $a_I|_{q=1}$ is a polynomial in $Q$.

As in the proof of Lemma~\ref{l.Acom}, we begin by showing that only finitely many $Q^d$ appear in the $q^{\geq 0}$ coefficients of the left-hand side of Equation~\eqref{e.monomial2}.  

Note that the first term of the left hand side gives $\prod_{k=1}^l P_{i_k}$. Suppose a $d> 0$ term  contributes to the $q^{\geq 0}$ coefficients on the left-hand side of Equation~\eqref{e.monomial2}, i.e. suppose that $q^{\sum_{k=1}^l d_{i_k}} J_d$ has non-positive order at $q=\infty$. Then applying Lemma~\ref{l.boundP} gives
\begin{align}\label{e.monomial-ineq}
        0\leq \sum_{k=1}^l  d_{i_k} - \nu_d&\leq \sum_{k=1}^l  d_{i_k}-m_{\hat{d}}
=\sum_{k=1}^l  \hat{d}_{i_k}-r(\hat{d}) -\frac{(\hat{d},\hat{d})}{2}.
\end{align}
Here, as in the proof of Lemma~\ref{l.Acom}, $\hat{d}$ is a lift of $d=d_P$ so that $m_{\hat{d}} = \min_{d_B^{\rm min} \leq d_B^+ \leq d_B}\{m_{d_B^+}\}$.  So $\hat{d}_i=d_i$ for $i\not\in I_P$.

There are finitely many possibilities for $d$ which satisfy the inequality \eqref{e.monomial-ineq}.  Indeed, the quadratic form ${(\;,\;)}$ is positive definite, so level sets of 
\[
\left(\sum_{k=1}^l \hat{d}_{i_k}-r(\hat{xd})\right) - \frac{(\hat{d},\hat{d})}{2}
\]
(as a function of $\hat{d}$) are ellipsoids in the vector space $\check\Lambda\otimes\mathbb{R}$.  It follows that the set 
\[
\left\{d=(d_j)_{j\not\in I_P}\Big| \left(\sum_{k=1}^l  \hat{d}_{i_k}-r(\hat{d})\right) - \frac{(\hat{d},\hat{d})}{2}\geq 0 \right\}
\]
is a bounded subset of $\check\Lambda^P\otimes\mathbb{R}$, so it can contain at most finitely many lattice points $d \in \check\Lambda^P_+$.   Therefore the left hand side of Equation~\eqref{e.monomial2} (and hence of Equation~\eqref{e.monomial3}) has finitely many $q^{\geq 0}$ terms.

Since $\TT = \mathrm{id} + O(q^{-1})$, we have
\[
  q^n Q^{d'}{\TT}(\Phi_w) = q^n Q^{d'} \Phi_w + (\text{terms involving }q^{n'}\text{ for }n'<n).
\]
In other words, the expansion of $q^n Q^{d'}{\TT}(\Phi_w)$ has a unique term with maximal power of $q$.  
Ordering the finitely many terms of the left hand side of Equation~\eqref{e.monomial2} according to the exponents of $q$, we may therefore use the elements
\[
  q^n Q^{d'}{\TT}(\Phi_w), \quad \text{ for } n\in \mathbb{Z}_{\geq 0},\; d'\in \check\Lambda^P_+,\; w \in W^P,
\]
to inductively remove the $q^{\geq 0}$ terms.

By Lemma~\ref{l.T-bound}, the operator $\TT$ satisfies the quadratic growth condition; it follows that for fixed $n$ and $w$, the element $q^n Q^{d'}{\TT}(\Phi_w)$ has only finitely many $q^{\geq 0}$ terms (essentially by repeating the argument given in the first part of this proof).
So the inductive removal of $q^{\geq0}$ terms ends after finitely many steps.\footnote{We stress that this step is {\em the only} part of our approach that uses bounds for $\mathsf{T}$.}  This means we can find {\em polynomials} $f_w\in F(T)[q,Q]$ so that the expression
\begin{equation}\label{e.claim}
  \MM:=\ \prod_{k=1}^l P_{i_k} \left(\sum_{d\geq 0} q^{\sum_{k=1}^l d_{i_k}} J_d Q^d \right)  -\sum_w {\TT}(f_w \Phi_w)
\end{equation}
vanishes at $q=+\infty$.  

From Equations~\ref{dq1} and ~\ref{dq2}, this is also equal to 
\begin{align*}
\MM=&\left(P^{\log Q/\log q}\right)^{-1}\left( \prod_{k=1}^l \qshift{{i_k}}\tilde{J}-\sum_w \tilde{\TT}(f_w \Phi_w) \right) \\
   &= \TT \left(\left( \prod \Aa_i \, \qshift{{i_k}} \right)(1)- \sum_w f_w\Phi_w \right) \\
   &=: \TT(\UU).
\end{align*}
By Lemma~\ref{l:T-infty}---which holds without change after the extension of scalars from $R(T)$ to $F(T)$---we conclude that $\MM=0$.
\end{proof}

In particular, the proof of Proposition~\ref{p.monomial} gives the following refinement of Equation~\eqref{e.monomial3}:
\[
\prod_{k=1}^l \qshift{{i_k}} \tilde{J}=\sum_w \tilde{\TT}(f_w \Phi_w)
\]
for \emph{polynomials} $f_w\in R(T)[q][Q]$, giving
\[ \prod_{k=1}^l P_{i_1}\star \cdots\star P_{i_l} = \sum_wf_w \Phi_w.\]

We now turn to our main theorem.  We fix an $R(T)$-basis $\{\Phi_w\}$ for $K_T(G/P)$, and recycle the notation to write $\Phi_w = \Phi_w\otimes 1$ for the corresponding $R(T)[[Q]]$-basis of $QK_T(G/P):= K_T(G/P) \otimes \mathbb{Z}[[Q]]$.

\begin{theorem}\label{t.finite}
The structure constants of $QK_T(G/P)$ with respect to the basis $\{\Phi_w\}$ are polynomials: they lie in the polynomial subring $R(T)[Q]$ of $R(T)[[Q]]$.
\end{theorem}

In particular, taking $\Phi_w$ to be a Schubert basis (of structure sheaves, canonical sheaves, or dual structure sheaves), we see that the quantum product of Schubert classes in $QK_T(G/P)$ is finite.

\begin{proof}
We begin by extending scalars from $R(T)$ to the fraction field $F(T)$ of $R(T)$, as in Proposition~\ref{p.monomial}; the structure constants are automatically in $R(T)[[Q]]$, so to prove they lie in $R(T)[Q]$, it is enough to show they lie in $F(T)[Q]$.

The assignment $P_{i_1} P_{i_2} \cdots P_{i_k}\mapsto P_{i_1}\star P_{i_2}\star \cdots \star P_{i_k}$ defines a ring homomorphism 
\begin{equation}\label{eqn:ring_hom1}
F(T)[P_1,\ldots, P_r; Q_1,\ldots,Q_{r}]\to F(T)\otimes_{R(T)}QK_T(G/P);
\end{equation}
let the kernel be $I$.  The resulting embedding of rings
\[
 F(T)[P_1,\ldots, P_{r}; Q_1,\ldots,Q_{r}]/I\hookrightarrow F(T)\otimes_{R(T)}QK_T(G/P)
\]
corresponds to the natural embedding of modules
\[
  F(T)\otimes_{R(T)}K_T(G/P)\otimes \mathbb{Z}[Q_1,\ldots,Q_r] \hookrightarrow F(T)\otimes_{R(T)}K_T(G/P)\otimes \mathbb{Z}[[Q_1,\ldots,Q_r]].
\]
It follows from Lemma~\ref{l.generate} that each  element $\Phi_w$ of the $R(T)$-basis for $K_T(G/P)$ 
can be written as a polynomial in $P_i$ with coefficients in $F(T)$. Therefore, each element $\Phi_w$ of the corresponding $R(T)[[Q]]$-basis for $QK_T(G/P)$  can be represented as a polynomial $\varphi_w=\varphi_w(P,Q)$ in $F(T)[P_1,\ldots, P_{r}][Q]$

The product of basis elements $\Phi_u\star\Phi_v$ in $QK_T(G/P)$ is given by a product $\varphi_u\, \varphi_v$ of polynomials in $P$ and $Q$, and by Proposition~\ref{p.monomial}, this product is a finite linear combination of classes in $F(T)\otimes_{R(T)}K_T(G/P)$ with coefficients in $\mathbb{Z}[Q]$.
\end{proof}

\appendix
\section{An inequality in the coroot lattice}\label{app:ineq}

Consider a root system (of finite type) in a real vector space $V$, with simple roots $\alpha_1,\ldots,\alpha_r$ and associated reflection group $W$.  Let $d=\sum_j d_j \alpha_j$ be an element of the root lattice, so the coefficients $d_j$ are integers.  Let $(\;,\;)$ be the $W$-invariant bilinear form on $V$, normalized so that $(\alpha_j,\alpha_j)=2$ for short roots.  Finally, let
\[
r(d) = \#\{j\,|\, d_j\neq 0\}.
\]
The purpose of this appendix is to prove a simple inequality.

\begin{lem*}
Assume that the root system contains no factors of type \rm{E}$_8$.  For any $i\in \{1,\ldots,r\}$, we have
\[
  \frac{(d,d)}{2} + r(d) \geq d_i,
\]
with equality if and only if $d=0$.
\end{lem*}

\begin{proof}
We may assume $r(d)=r$, i.e., $d$ has full support, since otherwise the problem reduces to a root subsystem.

Let us introduce a new variable $z$, and consider the quadratic form
\[
 Q(d_1,\ldots,d_r,z) = \frac{(d,d)}{2} - d_i z + r z^2.
\]
We will show that $Q$ is positive definite.  The lemma follows, by evaluating at $z=1$.

Let us write $A_Q$ for the symmetric matrix corresponding to $Q$, $A_R$ for the matrix corresponding to $\frac{1}{2}(\;,\;)$, and $A_{R(i)}$ for the matrix of the subsystem obtained by removing $\alpha_i$.  By reordering the roots as needed, we can assume $A_R$ and $A_{R(i)}$ are principal submatrices of $A_Q$, so $2A_Q$ has the form
\[
2A_Q = \left(\begin{array}{ccc|c}  &   &   & 0 \\  & 2A_R &  & \vdots \\  &  &  & -1  \\ \hline 0 & \cdots & -1 & 2r\end{array}\right)
\]
We see
\[
  \det(2 A_Q) = 2r\,\det(2 A_R) - \det( 2 A_{R(i)} ).
\]
To prove that $Q$ is positive definite, it suffices to check this determinant is positive, since we already know $A_R$ is positive definite.  This is easily done with a case-by-case check, using the data in Table~\ref{tab.1}.  
(Cf.~\cite[\S2.4]{h}, noting that our matrices are multiplied by factors corresponding to long roots.)
\begin{table}[ht]
\begin{tabular}{|c||c|c|c|c|c|c|c|c|} \hline
  $R$          &  A$_n$  & B$_n$ & C$_n$  & D$_n$ & E$_6$ & E$_7$ & F$_4$  & G$_2$   \\ \hline 
$\det(2 A_R)$  &  $n+1$  &  $2^n$  &  $4$     &  $4$ & $3$ & $2$  &   $4$     &  $3$ \\ \hline
\end{tabular}
\caption{Determinants for root systems \label{tab.1} }
\end{table}
\end{proof}

\begin{remark*}
In type E$_{8}$, if $i$ corresponds to the vertex of degree $3$ (the ``fork'') in the Dynkin diagram, then the quadratic form $Q$ is not positive definite: in fact, the determinant $\det(2 A_Q)$ is negative in this case.
\end{remark*}

\newpage
\section{Finiteness and quadratic growth in quantum $K$-theory}
\begin{center} 
	by Hiroshi Iritani\footnote{H.I.~is supported in part by JSPS KAKENHI Grant Number 16K05127, 16H06335, 16H06337 and 17H06127. \\ 
	Department of Mathematics, Kyoto University, Kitashirakawa-Oiwake-cho, Sakyo-ku, Kyoto, 606-8502, Japan  \\ 
	{\em E-mail address}: {\tt iritani@math.kyoto-u.ac.jp} 
	}
\end{center} 
\label{app:iritani}
%

\setcounter{equation}{0}
\renewcommand{\theequation}{B.\arabic{equation}}

We show that a quadratic growth condition for the zero orders of 
the fundamental solution $\sfT$ at $q=\infty$ 
is equivalent to the finiteness of the $q$-shift connection $\sfA$ 
associated with nef classes. 

Let $X$ be a smooth projective variety. 
Let $K(X)$ be the topological $K$-group with complex coefficients. 
We fix a basis $\{\Phi_\alpha\}$ of $K(X)$. 
Let $g$ denote the pairing on $K(X)$ given by 
$g(E,F) = \chi(E\otimes F)$.  
Let $\{\Phi^\alpha\}$ denote the dual basis with respect to 
the pairing $g$. 
Let $\sfT$ denote the fundamental solution of the quantum 
difference equation, defined by 
\[
\sfT(\Phi_\alpha) = \Phi_\alpha + 
\sum_{\substack{d\in \Eff(X) \\ d\neq 0}} \sum_\beta 
\corr{\Phi_\alpha, \frac{\Phi_\beta}{1-qL}}_{0,2,d} 
Q^d \Phi^\beta.  
\]
where $\Eff(X) \subset H_2(X,\Z)$ denotes the monoid 
generated by effective curves. 
We write $\sfT = \sum_{d\in \Eff(X)} \sfT_d Q^d$ 
with $\sfT_d \in \End(K(X))$.  
We say that \emph{$\sfT$ satisfies the quadratic growth condition}  
when the following holds: 
\begin{align} 
\label{eq:qgc}
\begin{split}  
& \begin{array}{l}  
\text{There exist a positive-definite inner product $(\cdot,\cdot)$ on 
$H_2(X)$, $m\in H^2(X)$} \\ 
\text{and a constant $c\in \R$ such that we have } 
\end{array} \\
& \qquad \qquad \qquad \qquad \quad 
\ord_{q=\infty} \sfT_d \ge \frac{1}{2}(d,d) + m\cdot d + c \\ 
& \begin{array}{l}  
\text{for all $d\in H_2(X)$, where $\ord_{q=\infty}$ is the order 
of zero at $q=\infty$.}  
\end{array} 
\end{split} 
\end{align} 
For a class $P\in K(X)$ of a line bundle,  
we write $p=- c_1(P)\in H^2(X)$ for the \emph{negative} of the 
first Chern class. 
For $p\in H^2(X)$, let $q^{p Q\partial_Q}$ denote the operator acting on 
power series in $Q$ as 
\[
q^{p Q\partial_Q}\left( \sum_{d \in H_2(X)} c_d Q^d \right) 
= \sum_{d\in H_2(X)} c_d q^{p\cdot d} Q^d.  
\]
The $q$-shift connection $\sfA$ associated 
with $P$ (or with $p= -c_1(P)$) is the operator 
\[
\sfA = \sfT^{-1} P q^{p Q\partial_Q}(\sfT)  
\]
where $P$ acts on $K(X)$ by the (classical) tensor product. 
The nontrivial fact is that $\sfA$ lies in the ring $\End(K(X))\otimes 
\C[q,q^{-1}][\![Q]\!]$, i.e.~it is a Laurent polynomial in $q$. 

\begin{prp*}
The fundamental solution 
$\sfT$ satisfies the quadratic growth condition \eqref{eq:qgc} 
if and only if the difference connections $\sfA$ 
associated with nef classes $p=-c_1(P)$ are polynomials in $Q$. 
\end{prp*}
\begin{proof} 
The `only if' statement was (essentially) proved by Anderson-Chen-Tseng 
\cite[Proposition 5]{act} (see also Proposition~\ref{p.monomial} above) although it was not phrased in 
this way. 
We give another proof for the convenience of the reader. 
We expand $\sfT^{-1} = (1+ \sum_{d\neq 0} \sfT_d Q^d)^{-1} 
= \sum_d \sfS_d Q^d$. Then:  
\[
\sfS_d = \sum_{k\ge 1}
\sum_{\substack{d(1)+\cdots+d(k) = d, \\ 
d(j)\in \Eff(X)\setminus \{0\}}} (-1)^k 
\sfT_{d(1)} \cdots \sfT_{d(k)} 
\]
for $d\neq 0$.
 
We claim that $\ord_{q=\infty} \sfS_d\to \infty$ as $|d| :=\sqrt{(d,d)} \to \infty$.  By the quadratic growth condition \eqref{eq:qgc} and the fact that $\ord_{q=\infty} \sfT_d \ge 1$ for $d\neq 0$, when $d=d(1)+\cdots+d(k)$ with $d(j) \in \Eff(X)\setminus \{0\}$, we have 
\begin{equation} 
\label{eq:est} 
\ord_{q=\infty}(\sfT_{d(1)} \cdots \sfT_{d(k)})  
\ge \max(k, f(d(1)) + \cdots + f(d(k))) 
\end{equation} 
where $f(d) := \frac{1}{2}(d,d) + m\cdot d + c$. 
Since $|d| \le |d(1)| + \cdots + |d(k)|$, 
there exists $i$ such that $|d(i)| \ge |d|/k$. 
Therefore if $k\le |d|^{\frac{1}{3}}$, then
\begin{align*} 
f(d(1)) + \cdots + f(d(k)) & = \frac{1}{2} 
\left(\sum_{i=1}^k (d(i),d(i))\right) + m \cdot d + ck \\
& \ge \frac{1}{2} \frac{|d|^2}{k^2} 
- |m| |d| - |c| k \\ 
& \ge \frac{1}{2} |d|^{\frac{4}{3}}  - |m| |d| - |c| |d|^{\frac{1}{3}}
\end{align*}  
Hence by \eqref{eq:est}, 
\[
\ord_{q=\infty} (\sfT_{d(1)} \cdots \sfT_{d(k)})  
\ge \min\left(
|d|^{\frac{1}{3}},  
\frac{1}{2} |d|^{\frac{4}{3}}  - |m| |d| - |c| |d|^{\frac{1}{3}} \right) 
\]
and the right-hand side diverges as $|d| \to \infty$.  
This proves the claim.

Let $\sfA$ be the $q$-shift operator associated with a nef 
class $p = -c_1(P)$.  Writing $\sfA = \sum_{d} \sfA_d Q^d$, we have 
\[
\sfA_d = \sum_{d'+d'' = d} \sfS_{d'} P q^{p\cdot {d''}} 
\sfT_{d''}.   
\]
Since $p$ is nef, $\sfA$ is regular at $q=0$ (see \cite[Proposition 2.10]{imt}).  On the other hand, using the quadratic growth condition \eqref{eq:qgc} 
again, we have 
\[
\ord_{q=\infty} \sfA_d \ge \min_{d'+d'' = d} 
(\ord_{q=\infty} \sfS_{d'} +f(d'') - p\cdot d'').  
\]
The right-hand side is positive for a sufficiently large $|d|$.  
In fact, both $N'= \{d' \in \Eff(X) : \ord_{q=\infty} \sfS_{d'}<0\}$ 
and $N'' = \{ d'' \in \Eff(X) : f(d'') - p\cdot d''<0\}$ are 
finite sets; when $d'\in N'$ and $d'+d''=d$, we have 
$f(d'') - p\cdot d''\to \infty$ as $|d|\to \infty$; 
similarly, when $d'' \in N''$ and $d'+d''=d$, we have 
$\ord_{q=\infty} \sfS_{d'}\to \infty$ as $|d| \to \infty$. 
Therefore $\sfA_d$ is regular at $q=0$ and 
$\ord_{q=\infty} \sfA_d>0$ for sufficiently large $|d|$. 
This implies that $\sfA_d = 0$ for sufficiently large $|d|$, i.e.~$\sfA$ is a polynomial in $Q$. 

Next we show the `if' statement. 
Suppose that all $q$-shift connections $\sfA$ 
associated with nef classes $p = -c_1(P)$ are polynomials in $Q$. 
Choose line bundles $P_1,\dots,P_k$ such that $p_i = -c_1(P_i)$ 
is nef and that $p_1,\dots,p_k$ form a basis of $H^2(X,\R)$. 
Let $\sfA^{(i)}$ be the $q$-shift connection associated with $P_i$.  
By assumption, there exists a finite set $F\subset \Eff(X) \setminus \{0\}$ 
of degrees such that $\sfA^{(i)}$ is expanded in the form: 
\[
\sfA^{(i)} = P_i + \sum_{d\in F} \sfA^{(i)}_d Q^d.  
\]
The fundamental solution $\sfT$ satisfies the $q$-difference 
equation $P_i q^{p_i Q\parfrac{}{Q}} \sfT = \sfT \sfA^{(i)}$, 
and therefore we have 
\begin{equation} 
\label{eq:diff_eq}
P_i q^{p_i \cdot d} \sfT_d = \sfT_d P_i 
+ \sum_{d' \in F} \sfT_{d-d'} \sfA^{(i)}_{d'}. 
\end{equation} 
Suppose $p_i\cdot d>0$. Then we have 
\[
\ord_{q=\infty} \sfT_d \ge 
p_i \cdot d + 
\min_{d'\in F}\left( \ord_{q=\infty} \sfT_{d-d'} \right)+C 
\] 
where $C:=\min_{1\le i\le k, d'\in F} (\ord_{q=\infty} \sfA^{(i)}_{d'})$. 
Note that the first term in the right-hand side of \eqref{eq:diff_eq} 
does not contribute to the vanishing order of $\sfT_d$ at $q=\infty$
because $p_i\cdot d>0$. 
Since this holds for all $i$ with $p_i \cdot d>0$, and 
there exists at least one $i$ with $p_i \cdot d>0$ 
when $d\in \Eff(X) \setminus \{0\}$ 
(note that $p_i \cdot d \ge 0$ since $p_i$ is nef), we have 
\begin{equation} 
\label{eq:recursive} 
\ord_{q=\infty} \sfT_d \ge 
\max_{1\le i\le k} \left( p_i \cdot d\right) + 
\min_{d'\in F} \left(\ord_{q=\infty} \sfT_{d-d'} \right)+C 
\end{equation} 
for all $d\in \Eff(X)\setminus \{0\}$. 
Introduce the norm $\|d\| := \sqrt{\sum_{i=1}^k (p_i \cdot d)^2}$ 
and set $B := \max_{d\in F} \|d\|$. 
Define the positive-definite inner product $(\cdot,\cdot)$ 
on $H_2(X)$ by 
\[
(d',d'') = \frac{1}{\sqrt{k} B}\sum_{i=1}^k (p_i \cdot d') (p_i \cdot d'').
\] 
Choose a class $m \in H^2(X)$ such that 
$m\cdot d\le C$ for all $d\in F$. This is possible since $F$ is a finite 
set contained in $\Eff(X) \setminus \{0\}$. 
We claim that 
\begin{equation} 
\label{eq:ind_hyp} 
\ord_{q=\infty} \sfT_d \ge \frac{1}{2} (d,d) + m\cdot d.  
\end{equation} 
This is true for $d=0$. We introduce a partial order $\prec$ in $\Eff(X)$ 
so that $d \prec d'$ if and only if $d'- d\in\Eff(X)$. 
Since every infinite descending chain 
$d(1)\succ d(2) \succ d(3) \succ \cdots$ in $\Eff(X)$ stabilizes, 
the induction argument works for this order. 
Suppose that $d_*\in \Eff(X)\setminus\{0\}$ and 
that \eqref{eq:ind_hyp} holds for all $d\in \Eff(X)$ with $d\prec d_*$.  
Using \eqref{eq:recursive} and the induction hypothesis, we have  
\begin{align*} 
\ord_{q=\infty} \sfT_{d_*} 
& \ge \max_{1\le i\le k} \left( p_i \cdot d_*\right) + 
\min_{d'\in F} \left( \frac{1}{2}(d_*-d',d_*-d') + m\cdot (d_*-d') \right) + C \\
& \ge \frac{1}{2}(d_*,d_*) + m\cdot d_* 
+ 
\max_{1\le i\le k} \left( p_i \cdot d_*\right)  
- \max_{d'\in F}(d_*, d') - \max_{d'\in F} (m \cdot d') +C \\ 
& \ge \frac{1}{2}(d_*,d_*) + m\cdot d_* 
+ \frac{1}{\sqrt{k}} \|d_*\| 
- \sqrt{(d_*,d_*)} \max_{d'\in F} \sqrt{(d',d')}  \\ 
& \ge \frac{1}{2}(d_*,d_*) + m\cdot d_* 
+ \frac{1}{\sqrt{k}} \|d_*\| 
- \frac{1}{\sqrt{k} B}  \|d_*\| \max_{d'\in F} \|d'\| \\ 
& \ge \frac{1}{2}(d_*,d_*) + m\cdot d_*.  
\end{align*} 
In the above computation, we used $\|d_*\| \le \sqrt{k} \max_{1\le i\le k} (p_i\cdot d_*)$. 
Hence the estimate \eqref{eq:ind_hyp} holds for $d_*$. 
The proposition is proved. 
\end{proof}

\begin{remark*} 
The Proposition holds also for the equivariant quantum $K$-theory.  The proof works verbatim. 
\end{remark*}


\begin{thebibliography}{12}

\bibitem{a} D.~Anderson, {\em Computing torus-equivariant K-theory of singular varieties}, in {\em Algebraic groups: structure and actions}, 1--15, Proc.~Sympos.~Pure Math., 94, Amer. Math. Soc., Providence, RI, 2017.

\bibitem{act} D.~Anderson, L.~Chen, and H.-H.~Tseng, {\em On the quantum K-ring of the flag manifold}, arXiv:1711.08414.

\bibitem{atiyah} M.~F.~Atiyah, \newblock{\em Elliptic Operators and Compact Groups}, Lecture Notes in Mathematics, Vol.~401, Springer-Verlag, 1974.

\bibitem{behrend} K.~Behrend, {\em Localization and Gromov-Witten invariants}, in Quantum cohomology (Cetraro, 1997), 3--38, Lecture Notes in Math., 1776, Fond. CIME/CIME Found. Subser., Springer, Berlin, 2002.

\bibitem{b-icm} A.~Braverman, {\em Spaces of quasi-maps into the flag varieties and their applications}, International Congress of Mathematicians. Vol. II, 1145--1170, Eur. Math. Soc., Z\"urich, 2006.

\bibitem{bdf} A.~Braverman, G.~Dobrovolska, and M.~Finkelberg, {\em Gaiotto-Witten superpotential and Whittaker D-modules on monopoles}, Adv. Math. 300 (2016), 451--472.

\bibitem{bf1} A.~Braverman and M.~Finkelberg, {\em Semi-infinite Schubert varieties and quantum K-theory of flag manifolds}, J. Amer. Math. Soc. 27 (2014), no. 4, 1147--1168.

\bibitem{bf2} A.~Braverman and M.~Finkelberg, {\em Twisted zastava and $q$-Whittaker functions}, J. Lond. Math. Soc. (2) 96 (2017), no. 2, 309--325.

\bibitem{bfg} A.~Braverman, M.~Finkelberg, and D.~Gaitsgory, {\em Uhlenbeck spaces via affine Lie algebras}, in {\em The unity of mathematics}, 17--135, Progr. Math., 244, Birkh\"auser, Boston, MA, 2006.

\bibitem{brion} M.~Brion, {\em Equivariant Chow groups for torus actions}, Transform. Groups 2 (1997), no. 3, 225--267.

\bibitem{bcmp1} A.~Buch, P.-E.~Chaput, L.~Mihalcea, and N.~Perrin, {\em Finiteness of cominuscule quantum K-theory}, Ann. Sci. \'Ec. Norm. Sup\'er. (4) 46 (2013), no. 3, 477--494.

\bibitem{bcmp2} A.~Buch, P.-E.~Chaput, L.~Mihalcea, and N.~Perrin, {\em Rational connectedness implies finiteness of quantum K-theory}, Asian J. Math. 20 (2016), no. 1, 117--122.

\bibitem{bcmp3} A.~Buch, P.-E.~Chaput, L.~Mihalcea, and N.~Perrin, {\em A Chevalley formula for the equivariant quantum K-theory of cominuscule varieties},  Algebr. Geom. 5 (2018), no. 5, 568--595, arXiv:1604.07500v2.

\bibitem{bm} A.~Buch and L.~Mihalcea, {\em Quantum K-theory of Grassmannians}, Duke Math.~J. 156 (2011), no. 3, 501--538.
 
\bibitem{cfks} I.~Ciocan-Fontanine, B.~Kim, and C.~Sabbah, {\em The abelian/nonabelian correspondence and Frobenius manifolds}, Invent. Math.  171 (2008), no. 2, 301--343.
 
\bibitem{e} P. Etingof, {\em Whittaker functions on quantum groups and q-deformed Toda operators},  in {\em Differential topology, infinite-dimensional Lie algebras, and applications}, 9--25, Amer. Math. Soc. Transl. Ser. 2, 194, Adv. Math. Sci., 44, Amer. Math. Soc., Providence, RI, 1999.

\bibitem{ffjmm}  B. Feigin, E. Feigin, M. Jimbo, T. Miwa, E. Mukhin, {\em Fermionic formulas for eigenfunctions of the difference Toda Hamiltonian}, Lett. Math. Phys. 88 (2009), no. 1-3, 39--77. 

\bibitem{g} A. Givental, {\em On the WDVV equation in quantum K-theory}, Dedicated to William Fulton on the occasion of his 60th birthday, Michigan Math. J. 48 (2000), 295--304.

\bibitem{gk} A. Givental and B. Kim. {\em Quantum cohomology of flag manifolds and Toda lattices}, Comm. Math. Phys. 168 (1995), 609--641.

\bibitem{gl} A. Givental and Y.-P. Lee, {\em Quantum K-theory on flag manifolds, finite-difference Toda lattices and quantum groups}, Invent. Math. 151 (2003), no. 1, 193--219.

\bibitem{gt} A. Givental and V. Tonita, {\em The Hirzebruch-Riemann-Roch Theorem in true genus-$0$ quantum K-theory}, in {\em Symplectic, Poisson, and Noncommutative Geometry}, 43--92, Math. Sci. Res. Inst. Publications, vol. 62, Cambridge Univ. Press, 2014, arXiv:1106.3136.

\bibitem{h} J.~Humphreys, {\em Reflection Groups and Coxeter Groups},  Cambridge Studies in Advanced Mathematics, 29. Cambridge University Press, Cambridge, 1990. xii+204 pp.

\bibitem{iim} T. Ikeda, S. Iwao, and T. Maeno, {\em Peterson isomorphism in K-theory and relativistic Toda lattice}, to appear in IMRN, arXiv:1703.08664v2. 

\bibitem{imt} H. Iritani, T. Milanov, and V. Tonita, {\em Reconstruction and convergence in quantum K-theory via difference equations}, Int. Math. Res. Not. IMRN 2015, no. 11, 2887--2937. 

\bibitem{k1} S. Kato, {\em Loop structure on equivariant  $K$-theory of semi-infinite flag manifolds}, arXiv:1805.01718v5.

\bibitem{k2} S. Kato, {\em Frobenius splitting of Schubert varieties of semi-infinite flag manifolds}, arXiv:1810.07106.
 
\bibitem{kollar} J.~Koll\'ar, {\em Singularities of the Minimal Model Program}, with the collaboration of S.~Kov\'acs, Cambridge Tracts in Mathematics, 200. Cambridge University Press, Cambridge, 2013. x+370 pp.

\bibitem{kontsevich} M.~Kontsevich, {\em Enumeration of rational curves via torus actions}, in {\em The moduli space of curves}, Birkh\"auser, 1995, 335--368.

\bibitem{kpsz} P. Koroteev, P. P. Pushkar, A. Smirnov, and A. M. Zeitlin, {\em Quantum K-theory of quiver varieties and many-body systems}, arXiv:1705.10419. 

\bibitem{kk} B.~Kostant and S.~Kumar, {\em $T$-equivariant K-theory of generalized flag varieties}, J. Differential Geom. 32 (1990), no. 2, 549--603.

\bibitem{kovacs} S.~Kov\'acs, {\em Irrational centers}, Pure Appl. Math. Q. 7 (2011), no. 4, Special Issue: In memory of Eckart Viehweg, 1495--1515.

\bibitem{llms} T. Lam, C. Li, L. C. Mihalcea, and M. Shimozono, {\em A conjectural Peterson isomorphism in K-theory}, J. Algebra, 513 (2018), 326--343, arXiv:1705.03435.

\bibitem{ls} T.~Lam and M.~Shimozono, {\em Quantum cohomology of $G/P$ and homology of affine Grassmannian}, Acta Math. 204 (2010), no. 1, 49--90.

\bibitem{l} Y.-P.~Lee, {\em Quantum K-theory. I. Foundations}, Duke Math. J. 121 (2004), no. 3, 389--424. 

\bibitem{lp} Y.-P.~Lee and R.~Pandharipande, {\em A reconstruction theorem in quantum cohomology and quantum K-theory}, Amer. J. Math. 126 (2004), no. 6, 1367--1379.

\bibitem{lm} C.~Lenart and T.~Maeno, {\em Quantum Grothendieck polynomials}, arXiv:0608232.

\bibitem{quart} G.~Quart, {\em Localization theorem in K-theory for singular varieties},  Acta Math. 143 (1979), no. 3-4, 213?217.

\bibitem{rossmann} W.~Rossmann, {\em Equivariant multiplicities on complex varieties}, \newblock{Ast\'erisque} 173--174 (1989), 11, 313--330.

\bibitem{s} A.~Sevostyanov, {\em Quantum deformation of Whittaker modules and the Toda lattice}, Duke Math. J. 105 (2000), no. 2, 211--238.

\bibitem{w} C.~Woodward, {\em On D.~Peterson's comparison formula for Gromov-Witten invariants of $G/P$}, Proc. Amer. Math. Soc.  133 (2005), no. 6, 1601--1609.

\end{thebibliography}
\end{document}